\documentclass{svmult}

\usepackage{calc,amscd,amsfonts,epsfig,psfrag,amsmath,amssymb,enumerate}
\usepackage{srcltx}
\usepackage{makeidx}         
\usepackage{graphicx}        
\usepackage{multicol}        
\usepackage[bottom]{footmisc}

\newtheorem{Theorem}{Theorem}[section] 
\newtheorem{Lemma}[Theorem]{Lemma} 
\newtheorem{Proposition}[Theorem]{Proposition} 
\newtheorem{Corollary}[Theorem]{Corollary} 

\def\ninf#1{ \| #1 \|_\infty }

\makeindex             

\begin{document}

\title*{Facilitated
  spin models: recent and new results}
\author{N. Cancrini\inst{1}\and F. Martinelli\inst{2}\and
  C. Roberto$^{\dag,}$\inst{3}\and C. Toninelli$^{\dag,}$\inst{4}
     }
\institute{Dip. Matematica Univ.l'Aquila, 1-67100 L'Aquila, Italy
\texttt{nicoletta.cancrini@roma1.infn.it}
\and Dip.Matematica, Univ. Roma Tre, Largo S.L.Murialdo 00146,
Roma,Italy \texttt{martin@mat.uniroma3.it}\and Universite Paris-est,
L.A.M.A. UMR 8050, 5 bd Descartes, 77454 Marne-la-Vall\'ee
France \texttt{cyril.roberto@univ-mlv.fr}\and Laboratoire de Probabilit\'es et Mod\`eles Al\`eatoires
  CNRS-UMR 7599 Universit\'es Paris VI-VII 4, Place Jussieu F-75252
  Paris Cedex 05 France \texttt{ctoninel@ccr.jussieu.fr}}
%
%

\maketitle

\footnotetext[1]{The material presented here is an
    expanded version of a series of lectures delivered by F.Martinelli at the 
    Prague summer school 2006 on \emph{Mathematical Statistical Mechanics}.
}\footnotetext[2]{$^\dag$ This work was partially
supported by gdre CNRS-INdAM 224, GREFIMEFI.}
\begin{abstract}  Facilitated or kinetically constrained spin models (KCSM) are a class of
  interacting particle systems reversible w.r.t. to a simple product
  measure. Each dynamical variable (spin) is re-sampled from
  its equilibrium distribution only if the surrounding
  configuration fulfills a simple local constraint which \emph{does not
  involve} the chosen variable itself. Such simple models are
  quite popular in the glass community since they 
  display some of the peculiar features of glassy
  dynamics, in particular they can undergo a dynamical arrest
  reminiscent of the liquid/glass transitiom. Due to the fact
  that the jumps rates of the Markov process can be zero, the whole
  analysis of the long time behavior becomes quite delicate and, until
  recently, KCSM have escaped a rigorous analysis with the
  notable exception of the East model. In these notes we will mainly
  review several recent mathematical results which, besides being
  applicable to a wide class of KCSM, have contributed to
  settle some debated questions arising in numerical simulations made by
  physicists. We will also provide some interesting new extensions. In
  particular we will show how to deal with interacting models reversible
  w.r.t. to a high temperature Gibbs measure and we will provide a
  detailed analysis of the so called one spin facilitated model on a
  general connected graph. \bigskip\noindent

{\bf Key words}: Glauber dynamics, spectral gap, kinetically constrained models,
dynamical phase transition, glass transition.
\end{abstract}

\section{Introduction and motivations}
Consider the following simple interacting particle system. At each site
of the lattice ${\mathbb Z}$ there is a dynamical variable $\sigma_x$, called in
the sequel ``spin'', taking values in $\{0,1\}$. With rate one each spin
attempts to change its current value by tossing a coin which lands head
with probability $p\in (0,1)$ and setting the new value to $1$ if head
and $0$ if tail. However the whole operation is performed only if the
current value on its right neighbor is $0$. Such a model is known under
the name of the East model \cite{JE} and it is easily checked to be
reversible w.r.t. the
product Bernoulli(p) measure. A
characteristic feature of the East model is that, when $q:=1-p\approx 0$, the relaxation to the reversible measure is
extremely slow \cite{noi}:
\begin{equation*}
  T_{\rm relax}\approx \left(1/q\right)^{\frac 12 \log_2(1/q)}
\end{equation*}
where $T_{\rm relax}$ is the inverse spectral gap in the spectrum of the (self-adjoint)
generator ${\mathcal L}$ of the process. Notice that if one writes
$p=\frac{e^{\beta}}{1+e^{\beta}}$ then $T_{\rm relax}\approx e^{c\beta^2}$
as $\beta\to \infty$,
a behavior that is refered to as a \emph{super-Arrhenius law} in the  physics literature. 

The East model is one of the simplest examples of a general class of
interacting particles models which are known in physical literature as 
{\sl facilitated} or {\sl kinetically constrained spin models} (KCSM). 

The common feature to all KCSM is that each dynamical variable, one for
each vertex of a connected graph ${\mathcal G}$ and with values
in a finite set $S$, waits an exponential time of 
mean one and then, if the
surrounding current configuration satisfies a simple local constraint,
is refreshed by sampling a new value from $S$ 
according to some apriori specified measure $\nu$. These models
have been introduced in the physical literature
\cite{FA1,FA2} to model the liquid/glass transition and more generally
the slow ``glassy'' dynamics which occurs in different systems (see
\cite{Ritort,TB} for recent review). In particular, they 
were devised to mimic the fact that the motion of a molecule in a dense
liquid can be inhibited by the presence of too many surrounding
molecules. That explains why, in all physical models, $S=\{0,1\}$ 
(empty or occupied site) and the
constraints specify the maximal number of particles (occupied sites)
on certain sites 
around a given one in order to allow creation/destruction on the
latter. As a consequence, the dynamics becomes increasingly
slow as the density of particles, $p$, is increased.  
Moreover there usually exist {\sl blocked
configurations}, namely configurations with all creation/destruction
rates identically equal to zero. This implies the existence of
several invariant measures (see \cite{Lalley} for a somewhat
detailed discussion of this issue in the context of the North-East
model), the occurrence of unusually long mixing times compared to
standard high-temperature stochastic Ising models and may induce the presence of ergodicity
breaking transitions without any counterpart at the level of the
reversible measure \cite{Cristina-leshouches}. 

Because of the presence of the constraints a mathematical analysis of
these models have been missing for a long time, with the notable
exception of the East model \cite{Aldous}, until a first recent
breakthrough \cite{noi,noi-JSTAT}. 

In this work we partly review
the results and the techniques of \cite{noi} but we also extend them in two
directions. Firstly we show that the main technique can be
adapted to deal with a weak
interaction among the variables obtained by replacing the reversible product measure with a general
high-temperature Gibbs measure. Secondly, motivated by some unpublished
considerations of D. Aldous \cite{Aldous2},  we analyze a special model, the so called
\emph{FA-1f} model, on a general connected graph and relate its
relaxation time to that of the East model.

\section{The models}
\label{general models}
\subsection{Setting and notation}
\label{setting}
The models considered here are defined on a locally finite, bounded
degree, connected graph
${\mathcal G}=(V,E)$ with vertex set  $V$ and edge set $E$. The associated graph distance
will be denoted by $d(\cdot,\cdot)$ and the degree of a vertex $x$ by
$\D_x$. The set of neighbors of $x$, {\it i.e. \ } $y\in V$ such that $d(y,x)=1$,
will be denoted by ${\mathcal N}_x$. For every subset $V'\subset V$ we
denote by $\partial V'$ the set of vertices in $V\setminus V'$ with one
neighbor in $V'$. In most cases the graph $G$ will either be
the $d$-dimensional lattice ${\mathbb Z}^d$ or a finite portion of it and in
both cases we need some additional notation that we fix now.
For any vertex $x\in {\mathbb Z}^d$ we define the $*$, thw oriented and the 
$*$-oriented
neighborhood of $x$ as
$$
\begin{array}{l}
{\mathcal N}^*_x=\{y\in {\mathbb Z}^d:\ y=x+\sum_{i=1}^d \alpha_i\vec e_i,\ \alpha_i=\pm 1,0
\ {\mbox {and}} \sum_i\alpha_i^2\neq 0\}\cr
  {\mathcal K}_x=\{y\in {\mathcal N}_x:\ y=x+\sum_{i=1}^d \alpha_i\vec e_i,\ \alpha_i\ge 0\}\cr
  {\mathcal K}^*_x=\{y\in {\mathcal N}^*_x:\ y=x+\sum_{i=1}^d \alpha_i\vec e_i,\ \alpha_i=1,0\}
\end{array}
$$
where $\vec e_i$ are the basis vactors of ${\mathbb Z}^d$.
Accordingly, the oriented and *-oriented neighborhoods $\partial_+ \Lambda,\ \partial_+^* \Lambda$ of a finite subset $\Lambda\subset
{\mathbb Z}^d$ are defined as
$
\partial_+\Lambda:=\left\{\cup_{x\in \Lambda}{\mathcal K}_x\right\}\setminus \Lambda,\
\partial_+^* \Lambda:=\left\{\cup_{x\in \Lambda}{\mathcal K}^*_x\right\}\setminus \Lambda.
$
A rectangle $R$ will be a set of sites of the form
$$
R:=[a_1,b_1]\times\dots\times [a_d,b_d]
$$
while the collection of finite subsets of ${\mathbb Z}^d$ will be denoted by
${\mathbb F}$.

\subsection{The probability space}
Let $\left(S,\nu\right)$ be a finite probability
space with $\nu(s)>0$ for any $s\in S$.  $G\subset S$
will denote a distinguished event in $S$, often referred to as the set
of ``good states'', and $q\equiv \nu(G)$ its probability.
 
Given $\left(S,\nu\right)$ we will consider the configuration space
$\Omega\equiv \Omega_V=S^{V}$ whose elements will be
denoted by Greek letters ($\omega,\eta\dots)$. If ${\mathcal G}'=(V',E')$ is a subgraph
of ${\mathcal G}$  and $\omega\in \Omega_V$ we will write $\omega_{V'}$ for its restriction
to $V'$. We will also say that a vertex $x$ is good for the
configuration $\omega$ if $\omega_x\in G$. 

On $\Omega$ equipped with the natural $\sigma$-algebra we will consider the
product measure $\mu:=\prod_{x\in V}\nu_x$, $\nu_x\equiv \nu$. If
${\mathcal G}'=(V',E')$ is a subgraph of ${\mathcal G}$ we will write $\mu_{V'}$ or
$\mu_{{\mathcal G}'}$ for the
restriction of $\mu$ to $\Omega_{V'}$. Finally,  for any $f\in
L^1(\mu)$, we will use the
shorthand notation $\mu(f)$ to denote its expected value and $\mathop{\rm Var}\nolimits(f)$ for its variance (when it exists).

\subsection{The Markov process}
\label{Markov preocess}
The general interacting particle models that will be studied here are
Glauber type Markov processes in $\Omega$, reversible w.r.t. the measure
$\mu$ and characterized by a finite
collection of \emph{influence classes} $\{{\mathcal C}_x\}_{x\in V}$, where ${\mathcal C}_x$ is just a
collection of subsets of $V$ (often of the neighbors of the vertex $x$) satisfying
the following general hypothesis:
\medskip\noindent
\begin{description}
\item[\bf Hp1] For all $x\in V$ and all $A\in
  {\mathcal C}_x$ the vertex $x$ does not belong to $A$.
\smallskip\noindent
\item[\bf Hp2] $r:=\sup_x\sup_{A\in {\mathcal C}_x}d(x,A)< +\infty$.
\end{description}
\medskip\noindent
In turn the influence classes together with the good event $G$ are the
key ingredients to define the constraints of each model.

\begin{definition}
\label{constraints}
Given a vertex $x\in V$ and a configuration $\omega$, we will say that the
constraint at $x$ is satisfied by $\omega$ if the indicator
\begin{equation}
\label{therates}
  c_{x}(\omega)=
  \begin{cases}
 1 & \text{if there exists a set $A\in
{\mathcal C}_x$ such that $\omega_y\in G$ for all $y\in A$}  \cr
0 & \text{otherwise}   
  \end{cases}
\end{equation}
is equal to one.
\end{definition}
\begin{remark}
The two general hypotheses above  tell us that in order to
check whether the constraint is satisfied at a given vertex we do not
need to check the current state of the \emph{vertex itself} and we only need to
check \emph{locally} around the vertex. This last requirement can actually be
weakened and indeed, in order to analyze certain spin exchange
kinetically constrained models \cite{noiconservativi}, a very efficient tool is to consider
\emph{long range} constraints !  
\end{remark}
The process that will be studied in the sequel can then be informally described
as follows. Each vertex $x$ waits an independent mean one exponential
time and then, provided that the current configuration $\omega$ satisfies
the constraint at $x$, the value $\omega_x$ is refreshed with a
new value in $S$ sampled from $\nu$ and the whole procedure starts again.

The generator ${\mathcal L}$ of the process can be constructed in a standard way
(see e.g. \cite{Liggett,Lalley}) and it is a non-positive self-adjoint
operator on $L^2(\Omega,\mu)$ with domain $Dom({\mathcal L})$ and Dirichlet form given by
$$
{\mathcal D}(f)=\sum_{x\in V}\mu\left(c_{x} \mathop{\rm Var}\nolimits_x(f)\right),\quad f\in Dom({\mathcal L})
$$
Here $\mathop{\rm Var}\nolimits_{x}(f)\equiv\int d\nu(\omega_x) f^2(\omega)
- \left(\int d\nu(\omega_x)f(\omega)\right)^2$ denotes the local variance with respect to
the variable $\omega_x$ computed while the other variables are held fixed.
To the generator ${\mathcal L}$ we can associate the Markov semigroup
$P_t:=e^{t{\mathcal L}}$ with reversible invariant measure $\mu$.

Notice that the constraints $c_x(\omega)$ are increasing functions w.r.t the
partial order in $\Omega$ for which $\omega\le\omega'$ iff $\omega'_x\in G$ whenever
$\omega_x\in G$.
However that does not imply in general that the process generated by
${\mathcal L}$ is attractive in the sense of Liggett \cite{Liggett}.

Due to the fact that in general the jump rates are not bounded away
from zero, irreducibility of the process is not guaranteed and the
reversible measure $\mu$ is usually not the only invariant measure
(typically there exist initial configurations that are blocked
forever). An
interesting question when ${\mathcal G}$ is infinite is therefore whether $\mu$ is ergodic/mixing for
the Markov process and whether there exist other ergodic stationary
measures. To this purpose it is useful to recall the following well
known result (see e.g. Theorem 4.13 in \cite{Liggett}).
\begin{Theorem}
\label{ergodic}
  The following are equivalent,
  \begin{enumerate}[(a)]
  \item $\lim_{t\to \infty} P_tf=\mu(f)$ in $L^2(\mu)$ for all
    $f\in L^2(\mu)$.
\item $0$ is a simple eigenvalue for ${\mathcal L}$.
  \end{enumerate}
\end{Theorem}
Clearly $(a)$ implies that $\lim_{t\to
  \infty}\mu\left(fP_tg\right)=\mu(f)\mu(g)$ for any $f,g\in L^2(\mu)$,
{\it i.e. \ } $\mu$ is mixing. 
\begin{remark}
  Even if $\mu$ is mixing there will exist in general infinitely many
  stationary measures, {\it i.e. \ } probability measures $\tilde \mu$ satisfying
  $\tilde\mu P_t=\tilde\mu$ for all $t\ge 0$. As an example, assume
  $c_x$ not identically equal to one and take an arbitrary probability
  measure $\tilde \mu$ such that $\tilde\mu\bigl(\{S\setminus
  G\}^{V}\bigr)=1$.  An interesting problem  is therefore to classify
  all the stationary ergodic measures $\tilde \mu$ of $\{P_t\}_{t\ge 0}$, where
  ergodicity means that $P_tf=f$ ($\tilde\mu$ a.e.) for all $t\ge 0$ implies that $f$ is
  constant ($\tilde \mu$ a.e.). As we will see later, when ${\mathcal G}={\mathbb Z}^2$ and for a
  specific choice of the constraint known as the North-East model, a
  rather detailed answer is now available \cite{Lalley}.
\end{remark}

When ${\mathcal G}$ is finite connected subgraph of an infinite graph
${\mathcal G}_\infty=(V_\infty,E_\infty)$, the ergodicity issue of the resulting
continuous time Markov chain can be attacked in two ways. 

The first one is to analyze the chain restricted to a suitably defined
ergodic component. Although such an approach is feasible and natural in
some cases (see section \ref{FA-1f general} for an example), the whole
analysis becomes quite cumbersome.  

Another possibility, which has several technical advantages over the
first one, is to unblock certain special vertices of ${\mathcal G}$ by relaxing
their constraints and restore irreducibility of the chain. A natural way to do
that is to imagine to extend the configuration $\omega$, apriori defined
only in $V$, to the vertices in $V_\infty\setminus V$ and to keep it
there frozen and equal to some reference configuration $\tau$ that will be
referred to as the \emph{boundary condition}. If enough vertices in
$V_\infty\setminus V$ are good for $\tau$, then enough vertices of ${\mathcal G}$
will become unblocked and the whole chain ergodic.

More precisely we can define the finite volume constraints with boundary
condition $\tau$ as
\begin{equation}
\label{theratesfinite}
c^\tau_{x,V}(\omega):= c_x(\omega\cdot\tau)
\end{equation}
where $c_x$ are the constraints for ${\mathcal G}_\infty$ defined in \eqref{therates} and $\omega\cdot
\tau\in\Omega$ denotes the configuration equal to $\omega$ inside $V$ and equal
to $\tau$ in $V_\infty\setminus V$.  Notice  that, for any $x\in V$, the
rate $c^\tau_{x,V}(\omega)$ (\ref{theratesfinite}) depends
on $\tau$ only through the indicators $\{{1 \mskip -5mu {\rm I}}_{\tau_z\in G}\}_{z\in {\mathcal B}}$,
where ${\mathcal B}$ is  the \emph{boundary set}
${\mathcal B}:=\left(V_\infty\setminus V\right) \cap \left(\cup_{z\in V}
  {{\mathcal C}_z}\right)$. Therefore,
instead of fixing $\tau$, it is enough to choose a subset ${\mathcal M}\subset
{\mathcal B}$, called the \emph{good boundary set}, and
define
\begin{equation}
\label{theratesfinite2}
c^{{\mathcal M}}_{x,V}(\omega):= c^\tau_{x,V}(\omega)
\end{equation}
where $\tau$ is \emph{any} configuration satisfying $\tau_z\in G$ for all
$z\in{\mathcal M}$ and $\tau_z\notin G$ for $z\in {\mathcal B}\setminus{\mathcal M}$.  We will say
that a choice of ${\mathcal M}$ is \emph{minimal} if the corresponding chain in
${\mathcal G}$ with the rates
(\ref{theratesfinite2}) is irreducible  and it is non-irreducible for any other choice
${\mathcal M} '\subset {\mathcal M}$. The choice ${\mathcal M}={\mathcal B}$ will be called
\emph{maximal}. For convenience we will write ${\mathcal L}_\Lambda^{\rm max}$
(${\mathcal L}_\Lambda^{\rm min}$) for the corresponding generators.
\begin{remark}
Without any other specification for the influence classes of the model
it may very well be the case that there exists no boundary conditions
for which the chain is irreducible and/or their existence may depend on
the choice of the finite subgraph ${\mathcal G}$.  However, as we will see later, for all the
interesting models discussed in the literature all these issues will have
a rather simple solution.   
\end{remark}
We will now describe some of the basic models and solve the problem of
boundary conditions for each one of them.

\subsection{0-1 Kinetically constrained spin models}
\label{models}
In most models considered in the physical literature the finite
probability space $(S,\nu)$ is a simple $\{0,1\}$ Bernoulli space and
the good set $G$ is conventionally chosen as the empty (vacant) state
$\{0\}$. Any model with these features will be called in the sequel a
``0-1 KCSM'' (kinetically constrained spin model). Although in most
cases the underlying graph ${\mathcal G}$ is a regular lattice like ${\mathbb Z}^d$,
whenever is possible we will try to work in full generality.

Given a 0-1 KCSM, the parameter $q=\nu(0)$ can be varied in $[0,1]$
while keeping fixed the basic structure of the model ({\it i.e. \ } the notion of
the good set and the constraints $c_x$) and
it is natural to define a critical value $q_c$ as
$$
q_c=\inf\{q\in[0,1]:\, 0 \text{ is a simple eigenvalue of ${\mathcal L}$}\}
$$  
As we will prove below $q_c$ coincides with the \emph{bootstrap
  percolation threshold} $q_{bp}$
of the model defined as follows \cite{Schonmann} \footnote{In most of
  the bootstrap percolation literature the role of the $0$'s and the
  $1$'s is inverted}. For any $\eta\in \Omega$
define the bootstrap map $T:\Omega\mapsto \Omega$ as  
\begin{equation}
  \label{eq:bootstrap map}
  (T\eta)_x=0 \quad \text{if either}\quad 
\eta_x=0 \quad \text{or}\quad c_x(\eta)=1.
\end{equation}
Denote by $\mu^{(n)}$ the probability
measure on $\Omega$ obtained by iterating $n$-times the above mapping
starting from $\mu$. As $n\to \infty$ $\mu^{(n)}$ converges to
a limiting measure $\mu^{(\infty)}$ \cite{Schonmann} and it is natural
to define the critical value $q_{bp}$ as 
$$
q_{bp}=\inf\{q\in [0,1]:\, \mu^{\infty}=\delta_0\}
$$ 
where $\delta_0$ is the probability measure assigning unit mass to the
constant configuration identically equal to zero. In other words
$q_{bp}$ is  the infimum
of the values $q$ such that, with probability one, the graph ${\mathcal G}$ can be
entirely emptied. 
Using the fact that the $c_x$'s are increasing function of $\eta$ it is
easy to check that  $\mu^{(\infty)}=\delta_0$ for any $q>q_{bp}$ .
\begin{Proposition}[\cite{noi}]
\label{qc=qbp} $q_c=q_{bp}$ and for any $q>q_c$
  $0$ is a simple eigenvalue for ${\mathcal L}$.
\end{Proposition}
\begin{remark}
In \cite{noi} the proposition has been proved in the special case
${\mathcal G}={\mathbb Z}^d$ but actually the same arguments apply to any bounded degree
connected graph.
\end{remark}
Having defined the bootstrap percolation it is natural to divide the 0-1
KCSM into two distinct classes.
\begin{definition}
\label{noncoop}
We will say that a 0-1 KCSM is \emph{non cooperative} if there exists a
finite set ${\mathcal B}\subset V$ such that any configuration $\eta$ which is empty in all the
sites of ${\mathcal B}$ reaches the empty configuration (all 0's) under
iteration of the bootstrap mapping. Otherwise the model will be called \emph{cooperative}.   
\end{definition}
\begin{remark}
Notice that for a non-cooperative model the critical value $q_c$ is obviously
zero since with $\mu$-probability one a configuration will contain the
required finite set ${\mathcal B}$ of zeros.  
\end{remark}
We will now illustrate some of the most studied models.

\hfill\break
{\bf [1] Frederickson-Andersen (FA-jf) facilitated models
  \cite{FA1,FA2}.} In the facilitated models the constraint at $x$ requires
that at least $j\le \D_x$ neighbors are vacant. More formally  
$$
{\mathcal C}_x=\{A\subset {\mathcal N}_x: |A|\ge j\}
$$
When $j=1$ the model is non-cooperative for any connected graph ${\mathcal G}$ and
ergodicity of the Markov chain is clearly guaranteed by the presence of
at least one unblocked vertex. When $j>1$ ergodicity on a general graph
is more delicate and we restrict ourselves to finite rectangles $R$ in
${\mathbb Z}^d$.  In that case and for the most constrained cooperative  case
$j=d$ among the irreducible ones, irreducibility is guaranteed if we assume a
boundary configuration identically empty on $\partial_+R$. Quite remarkably, using results
from bootstrap percolation \cite{Schonmann} combined with proposition
\ref{qc=qbp}, when ${\mathcal G}={\mathbb Z}^d$ and $2\le j\le d$ the ergodicity
threshold $q_c$ always vanishes. 

\hfill\break
{\bf [2] Spiral model \cite{spiral1,spiral2}} This model is defined on ${\mathbb Z}^2$
with the following choice for the influence classes
$$
{\mathcal C}_x=\{NE_x\cup SE_x;~SE_x\cup SW_x;~SW_x\cup NW_x;~NW_x\cup NE_x\}
$$
where $NE_x=(x+\vec e_2,x+\vec e_1+\vec e_2)$, $SE_x=(x+\vec e_1,x+\vec
e_1-\vec e_2)$, $SW_x=(x-\vec e_2,x-\vec e_2-\vec e_1)$ and
$NW_x=(x-\vec e_1;x-\vec e_1+\vec e_2)$. In other words the vertex $x$ can
flip iff either its North-East ($NE_x$) or its South-West
($SW_x$) neighbours (or both of them) are empty {\sl and} either its
North-West ($NW_x$) or its South-East ($SE_x$) neighbours (or both of
them) are empty too.  The model is clearly cooperative and in
\cite{spiral2} it has been proven that its critical point $q_c$
coincides with $1-p_c^o$, where $p_c^o$ is the critical treshold for
oriented percolation. The interest of this model lies on the fact that
its bootstrap percolation is expected to display a peculiar mixed
discontinuous/critical character which makes it relevant as a model
for the liquid glass and more general jamming transitions
\cite{spiral1,spiral2}.

\hfill\break
{\bf [3] Oriented models.}
Oriented models are similar to the facilitated models but the neighbors of a given vertex $x$ that must be vacant
in order for $x$ to become free to flip, are chosen according to some
orientation of the graph. Instead of trying to describe a very general
setting we present three important examples. 

\begin{example}
The first and best known example is the so called East model \cite{JE}. Here ${\mathcal G}={\mathbb Z}$ and for every $x\in {\mathbb Z}$
the influence class
${\mathcal C}_x$ consists of the vertex $x+1$. In other words any vertex can flip iff its right neighbor
is empty. The minimal boundary conditions in a finite interval which
ensure irreducibility of the chain are of
course empty right boundary, {\it i.e. \ } the rightmost vertex is always
unconstrained. The model is clearly cooperative but $q_c=0$ since in
order to empty ${\mathbb Z}$ it is enough to start from a configuration for
which any site $x$ has some empty vertex to its right. One could easily
generalize the model to the case when ${\mathcal G}$ is a rooted tree (see
section \ref{FA-1f general}). In that
case any vertex different from the root can be updated iff its ancestor
is empty. The root itself is unconstrained.
\end{example}

\begin{example}
The second example is the North-East model in ${\mathbb Z}^2$ \cite{RJM}.
Here one chooses ${\mathcal C}_x$ as the North and East neighbor of $x$. 
The model is clearly cooperative and its critical point $q_c$ coincides with $1-p^o_c$,
where $p_c^o$ is the critical threshold for oriented percolation in
${\mathbb Z}^2$ \cite{Schonmann}. For such a model much more can be said about
the stationary ergodic measures of the Markov semigroup $P_t$. 
\begin{Theorem}[\cite{Lalley}] If $q<q_c$ the trivial measure $\delta_1$ that
  assigns unit mass to the configuration identically equal to $1$ is the
  only translation invariant, ergodic, stationary measure for the
  system. If $q\ge q_c$ the reversible measure $\mu$ is the unique, non
  trivial, ergodic, translation invariant, stationary measure. 
\end{Theorem}
\end{example}

\begin{example}
The third model was suggested in \cite{Aldous} and it is defined on a rooted
(finite or infinite)
binary tree ${\mathcal T}$. Here a vertex $x$ can flip iff its two children are
vacant. If the tree is finite then ergodicity requires that all the
leaves of ${\mathcal T}$ are unconstrained. It is easy to check that the critical
threshold satisfies $q_c=1/2$, the site percolation threshold on the binary tree.  
\end{example}

\section{Quantities of interest and related problems}
\label{main questions}
Back to the  general model we now define two main quantities that are of
mathematical and physical interest.

The first one is the spectral
gap of the generator ${\mathcal L}$, defined as 
\begin{equation}
  \label{eq:gap}
\mathop{\rm gap}\nolimits({\mathcal L}):=\inf_{f\neq \text{const}}\frac{{\mathcal D}(f)}{\mathop{\rm Var}\nolimits(f)}
\end{equation}
A positive spectral gap implies that the reversible measure $\mu$ is
mixing for the semigroup $P_t$ with exponentially decaying correlations:
$$
\mathop{\rm Var}\nolimits\left(P_t f\right)\leq e^{-2t\mathop{\rm gap}\nolimits({\mathcal L})}\mathop{\rm Var}\nolimits(f),\qquad \forall \, f\in
L^2(\mu).
$$
\begin{remark}
  In the sequel the time scale $T_{\rm rel}:=\mathop{\rm gap}\nolimits^{-1}$ which is naturally fixed by
  the spectral gap will be refered to as
  the \emph{relaxation time} of the process. 
\end{remark}
For a $0$-$1$ KCSM, two natural questions arise.
\begin{enumerate}
\item Define the new critical point 
$
q'_c:=\inf\{q\in [0,1]:\ \mathop{\rm gap}\nolimits({\mathcal L})>0\}.
$
Obviously $q'_c\ge q_c$. Is it the case that equality holds ? 
\item If $q'_c=q_c$ what is the behaviour of $\mathop{\rm gap}\nolimits({\mathcal L})$ as $q\downarrow q_c$ ? 
\end{enumerate}
As we will see later for most of the relevant models it is possible to
answer in rather detailed way to both questions.

The second quantity of interest is the so called \emph{persistence
  function} (see e.g. \cite{Ha,SE}) defined by 
\begin{equation}
  \label{eq:Pers}
F(t):=\int d\mu(\eta)\; {\mathbb P}(\sigma^\eta_0(s)=\eta_0,\; \forall s\le t)
\end{equation}
where $\{\sigma^\eta_s\}_{s\ge 0}$ denotes the process started from the
configuration $\eta$. In some sense the persistence function, a more
accessible quantity to numerical simulation than the spectral gap, provides a
measure of the ``mobility'' of the system. Here the main questions are:
\begin{enumerate}
\item What is the behavior of $F(t)$ for large time
  scales ?
\item For a $0$-$1$ KCSM is it the case that $F(t)$ decays exponentially
  fast as $t\to \infty$ for any $q>q'_c$ ? 
\item If the answer to the previous question is positive, is the decay rate
  related to the spectral gap in a simple way or the decay rate of
  $F(t)$ requires a deeper knowledge of the spectral density of ${\mathcal L}$ ?  
\item Is it possible to exhibit examples of $0$-$1$ KCSM in which the
  persistence function shows a crossover between a stretched and a pure exponential
  decay ?  
\end{enumerate}
Unfortunately the above questions are still mostly unanswered except for
the first two.   
\subsection{Some useful observations to bound the spectral gap}
It is important to observe the following kind of monotonicity that can be
exploited in order to bound the spectral gap of one model with the
spectral gap of another one.
\begin{definition}
Suppose that we are given two
influence classes ${\mathcal C}_0$ and ${\mathcal C}'_0$, denote by
$c_x(\omega)$ and $c'_x(\omega)$ the corresponding rates and by ${\mathcal L}$ and ${\mathcal L}'$
the associated generators on $L^2(\mu)$. If,  for all $\omega\in \Omega$ and all
$x\in V$,
$c'_x(\omega)\le c_x(\omega)$, we say that ${\mathcal{L}}$ is dominated by ${\mathcal{L'}}$.
\end{definition}
\begin{remark}
The term domination here has the same meaning it has in the context of
bootstrap percolation. It means that the KCSM associated to ${\mathcal L}'$ is
\emph{more} constrained than the one associated to ${\mathcal L}$.
\end{remark}
Clearly, if ${\mathcal L}$ is dominated by ${\mathcal L}'$, ${\mathcal D}'(f)\leq {\mathcal D}(f)$ and
therefore $\mathop{\rm gap}\nolimits({\mathcal L}')\le \mathop{\rm gap}\nolimits({\mathcal L})$.
\begin{example}
Assume that the graph ${\mathcal G}$ has $n$ vertices and contains a Hamilton
path $\Gamma=\{x_1,x_2,\dots,x_n\}$, {\it i.e. \ } $d(x_{i+1},x_i)=1$ for all
$1\le i\le n-1$ and $x_i\neq x_j$ for all $i\neq j$. Consider the FA-1f model
on ${\mathcal G}$ with one special vertex, e.g. $x_n$, unconstrained ($c_{x_n}\equiv 1$). Then,
if we replace ${\mathcal G}$ by $\Gamma$ equipped with its natural graph structure and we
denote by ${\mathcal L}$ and  ${\mathcal L}'$ the respective generators,
we get that $\mathop{\rm gap}\nolimits( {\mathcal L})\ge \mathop{\rm gap}\nolimits({\mathcal L}')$. Clearly ${\mathcal L}'$ describes the FA-1f model on
the finite interval $[1,\dots,n]\subset {\mathbb Z}$ with the last vertex free to
flip. This in turn is dominated by ${\mathcal L}_{\rm East}$, the
generator of the East model on $[1,\dots,n]$, which is
 known to have
a positive \cite{Aldous,noi} spectral gap uniformly in $n$. Therefore the latter result holds also for $\mathop{\rm gap}\nolimits({\mathcal L}')$ and $\mathop{\rm gap}\nolimits({\mathcal L})$ .
\end{example}
\begin{example}
Along the lines of the previous example we
could lower bound the spectral gap of the FA-2f model in ${\mathbb Z}^d$, $d\ge
2$, with that in ${\mathbb Z}^2$, by restricting the sets $A\in {\mathcal C}_0$ to
e.g. the $(\vec e_1, \vec e_2)$-plane.  
\end{example}  
For a last and more detailed example of the comparison technique we refer the reader
to section \ref{FA-1f general}. 

Although the comparison technique can be quite effective in proving positivity
of the spectral gap, one should keep in mind that, in general, it provides quite poor bounds,
particularly in the limiting case $q\downarrow q_c$.  

The second observation we make consists in relating $\mathop{\rm gap}\nolimits({\mathcal L})$ when the
underlying graph is infinite to its finite graph analogue. Fix $r\in V$
and let ${\mathcal G}_{n,r}\subset {\mathcal G}$ be the connected ball centered at $r$ of
radius $n$. Suppose that $\inf_{n}\mathop{\rm gap}\nolimits({\mathcal L}^{\rm max}_{{\mathcal G}_{n,r}})>0$. It
is then easy to conclude that $\mathop{\rm gap}\nolimits({\mathcal L})>0$.

Indeed, following Liggett Ch.4 \cite{Liggett}, for any $f\in Dom({\mathcal L})$
with $\mathop{\rm Var}\nolimits(f)>0$ pick $f_n\in L^2(\Omega,\mu)$ depending only on finitely
many spins so that $f_n \to f$ and ${\mathcal L} f_n \to {\mathcal L} f$ in $L^2$. Then
$\mathop{\rm Var}\nolimits(f_n)\to \mathop{\rm Var}\nolimits(f)$ and ${\mathcal D}(f_n)\to {\mathcal D}(f)$. But since $f_n$ depends
on finitely many spins
$$
\mathop{\rm Var}\nolimits(f_n)=\mathop{\rm Var}\nolimits_{{\mathcal G}_{m,r}}(f_n)\quad \text{and}\quad {\mathcal D}(f_n)={\mathcal D}_{{\mathcal G}_{m,r}}(f_n)
$$ 
provided that $m$
is a large enough square (depending on $f_n$). Therefore 
$$
\frac{{\mathcal D}(f)}{\mathop{\rm Var}\nolimits(f)}\ge \inf_{n}\mathop{\rm gap}\nolimits({\mathcal L}_{{\mathcal G}_{n,r}})>0.
$$  
and $\mathop{\rm gap}\nolimits({\mathcal L})\ge \inf_{n}\mathop{\rm gap}\nolimits({\mathcal L}_{{\mathcal G}_{n,r}})>0$.

\section{Main results for 0-1 KCSM on regular lattices}
In this section we state some of the main results for a general 0-1 KCSM
on ${\mathbb Z}^d$  which have been obtained in \cite{noi}.

Fix an integer length scale $\ell$ larger than the range of the constraints and let
${\mathbb Z}^d(\ell)\equiv \ell\,{\mathbb Z}^d$. Consider a
partition of ${\mathbb Z}^d$ into disjoint rectangles $\Lambda_z:=\Lambda_0+z$, $z\in
{\mathbb Z}^d(\ell)$, where $\Lambda_0=\{x\in{\mathbb Z}^d:\ 0\le x_i\le \ell-1,\; i=1,..,d\}$.
\begin{definition}
\label{defgelle}
Given $\epsilon\in (0,1)$ we say that $G_{\ell}\subset \{0,1\}^{\Lambda_0}$ is
a $\epsilon$-good set of configurations on scale $\ell$ if the following
two conditions are satisfied:
\begin{enumerate}[(a)]
\item $\mu(G_{\ell})\ge 1-\epsilon$.
\item For any collection $\{\xi^{(x)}\}_{x\in {\mathcal K}_0^*}$ of spin
  configurations such that $\xi^{(x)}\in G_\ell$ for all $x\in
  {\mathcal K}_0^*\,$, the following holds.  For any $\xi\in \Omega$ which coincides
  with $\xi^{(x)}$ in $\cup_{x\in {\mathcal K}_0^*}\Lambda_{\ell x}$, there exists a
  sequence of legal moves inside $\cup_{x\in {\mathcal K}_0^*}\Lambda_{\ell x}$ ({\it i.e. \ }
  single spin moves compatible with the constraints) which transforms
  $\xi$ into a new configuration $\tau\in\Omega$ such that the Markov chain in
  $\Lambda_0$ with boundary conditions $\tau$ is ergodic.
\end{enumerate}
\end{definition}
\begin{remark} 
\label{perspiral}
In general the transformed configuration $\tau$ will be
  identically equal to zero on $\partial_+^* \Lambda_0$. It is also clear
  that assumption (b) has been made having in mind models
  like the East, the FA-jf or the N-E which, modulo rotations, are dominated by
  a model with influence class  $\tilde {\mathcal C}_x$ \emph{entirely
  contained} in the sector $\{y: y=x+\sum_{i=1}^d \alpha_i \vec e_i,\ \alpha_i\ge 0\}$.
If this is not the case one should instead
use a non rectangular geometry for the tiles
of the partition of ${\mathbb Z}^d$, adapted
to the choice of the influence classes.
For example for the Spiral Model the basic tile at lenght scale $\ell$
is a quadrangular region ${\cal{R}}_{0}$ with one side parallel to $\vec
e_1$ and two sides parallel to $\vec e_1+\vec e_2$,
${\cal{R}}_{0}:=\cup_{1}^{\ell} S_0+(i-1)(\vec e_1+\vec e_2)$ with
$S_0:=\{x\in{\mathbb Z}^2: 0\leq x_1\leq \ell-1, x_2=0\}$. In this case
condition (b) should also be modified by substituting 
everywhere $\partial_+^* {\Lambda}_0$ with $\widetilde \partial_+^* {\Lambda}_0:=\vec e_1,\vec e_1-\vec e_2,\vec -e_2$. 
\end{remark}
With the above notation the first main result of \cite{noi} can be formulated as follows.
\begin{Theorem}
There exists a universal constant $\epsilon_0\in (0,1)$ such that, if
there exists $\ell$ and a $\epsilon_0$-good set $G_\ell$ on scale $\ell$,
then $\inf_{\Lambda\in {\mathbb F}}\mathop{\rm gap}\nolimits({\mathcal L}^{\rm max}_\Lambda)>0$. In particular $\mathop{\rm gap}\nolimits({\mathcal L})>0$.
\label{main theorem 01}
\end{Theorem}
In several examples, e.g. the FA-jf models, the natural candidate for the event
  $G_\ell$ is  the event that the tile $\Lambda_0$ is ``internally
    spanned'', a notion borrowed from bootstrap percolation
  \cite{Aizenman,Schonmann,Cerf,Holroyd,Cerf2}:
  \begin{definition}
\label{ISDEF}
We say that a finite set $\Gamma\subset{\mathbb Z}^d$ is \emph{internally spanned}
by a configuration $\eta\in\Omega$ if, starting from
the configuration $\eta^\Gamma$ equal to one outside $\Gamma$ and equal to $\eta$
inside $\Gamma$,  there exists a sequence of legal moves
  inside $\Gamma$ which connects
  $\eta^\Gamma$ to the configuration identically equal to zero inside $\Gamma$ and
  identically equal to one outside $\Gamma$.
  \end{definition}
  Of course whether or not the set $\Lambda_0$ is internally spanned for $\eta$
  depends only on the restriction of $\eta$ to $\Lambda_0$. One of the major
  results in bootstrap percolation problems has been the exact
  evaluation of the $\mu$-probability that the box $\Lambda_0$ is internally
  spanned as a function of the length scale $\ell$ and the parameter $q$
  \cite{Holroyd, Schonmann, Cerf,Cerf2, Aizenman}.  For non-cooperative
  models it is obvious that for any $q>0$ such probability tends very
  rapidly (exponentially fast) to one as $\ell\to \infty$, since the
  existence of at least one completely empty finite set ${\mathcal B}+x\subset
  \Lambda_0$ (see definition \ref{noncoop}),
  allows to empty all $\Lambda_0$.
  For some cooperative systems like e.g. the FA-2f in ${\mathbb Z}^2$, it has been shown that for any $q>0$ such
  probability tends very rapidly (exponentially fast) to one as $\ell\to
  \infty$ and that it abruptly jumps from being very small to being
  close to one as $\ell$ crosses a critical scale $\ell_c(q)$. In most
  cases the critical length $\ell_c(q)$ diverges very rapidly as
  $q\downarrow 0$. Therefore, for such models and $\ell>\ell_c(q)$, one
  could safely take $G_\ell$ as the collection of configurations $\eta$
  such that $\Lambda_0$ is internally spanned for $\eta$. We now formalize what
  we just said.
  \begin{Corollary}
\label{IS}
  Assume that $\lim_{\ell\to \infty}\mu(\Lambda_0 \text{
    is internally spanned })=1$ and that the Markov chain in $\Lambda_0$ with
  zero boundary conditions on $\cup_{x\in {\mathcal K}_0^*}\Lambda_{\ell x}$ is ergodic. Then $\mathop{\rm gap}\nolimits({\mathcal L})>0$.
  \end{Corollary}

We stress that for some models a notion of good event which differs
from requiring internal spanning is needed. This is the case for the N-E and Spiral models, as can be immediately seen by noticing that at any length scale it is possible to construct small clusters of particles in proper corners of the tiles that can never be erased by internal moves. The choice of the proper $\epsilon$-good set of confugurations for N-E has already been discussed in \cite{noi}. For the Spiral Model the  definition which naturally arises from the results in \cite{spiral1} is the following. Let  $\widetilde{\cal{R}}_{0}$ be the region obtained from ${\cal{R}}_{0}$ by subtracting two proper quadrangular regions  at the bottom left and top right corners, namely $\widetilde{\cal{R}}_{0}:={\cal{R}}_{0}\setminus ({\cal{R}}_{bl}\cup {\cal{R}}_{tr})$ where ${\cal{R}}_{bl}$ (${\cal{R}}_{tr}$) have the same shape of ${\cal{R}}_{0}$ shrinked at length scale $\ell/4$ and have the bottom left (top right) corner which coincides with the one of ${\cal{R}}_{0}$.
The $\epsilon$-good set of configurations on scale $\ell$,  $G_{\ell}$,
includes all configurations $\eta$ such that there 
exists a sequence of legal moves inside ${\cal{R}}_{0}$ which connects
$\eta^{{\cal{R}}_{0}}$ (the configuration which has all ones outside
${\cal{R}}_{0}$ and equals $\eta$ inside)
 to a configuration 
identically equal to zero inside $\widetilde{\cal{R}}_{0}$. Lemma 4.7 and Proposition 4.9 of \cite{spiral1} prove, respectively, property (a) and (b) of Definition \ref{defgelle} (with $\partial_+^* {\Lambda}_0$ substituted with $\vec e_1,\vec e_1-\vec e_2,\vec -e_2$, see remark \ref{perspiral}) when the density is below the critical density of oriented percolation. Thus, using this definition for the good event and Theorem \ref{main theorem 01} we conclude that 
\begin{Theorem}
$\mathop{\rm gap}({\cal {L}}_{\mbox{spiral}})>0$ at any $\rho< p_c^o$.
\end{Theorem}
The second main result concerns the long time behavior of the
persistence function $F(t)$ defined in \eqref{eq:Pers}.

\begin{Theorem}
\label{persf}
Assume that $\mathop{\rm gap}\nolimits({\mathcal L})>0$. Then $F(t)\le
e^{-q \mathop{\rm gap}\nolimits t} + e^{-p\mathop{\rm gap}\nolimits t}$. 
\end{Theorem}
\begin{remark}
\label{remarkpers}
The above theorems disprove some conjectures which appeared in the physics literature
\cite{GPG,Ha,BG,WBG1}, based on numerical simulations and approximate
analytical treatments, on the existence of a second critical point
$q_c'>q_c$ at which the spectral gap vanishes and/or below which
$F(t)$ would decay in a stretched exponential form
$\simeq \exp(-t/\tau)^{\beta}$ with $\beta<1$.

Theorem \ref{persf} also indicates that one can obtain \emph{upper bounds} on
the spectral gap by proving \emph{lower bounds} on the persistence
function. Concretely a lower bound on the persistence function can be obtained
by restricting the $\mu$-average to those initial configurations
$\eta$ for which the origin is blocked with high probability for all times
$s\le t$.  Unfortunately in most models such a strategy leads to lower
bound on $F(t)$ which are usually quite far from the above upper bound
and it is an interesting open problem to find an exact asymptotic as
$t\to \infty$ of $F(t)$. 

Finally we observe that for the North-East model on ${\mathbb Z}^2$ at the
critical value $q=q_c$ the spectral gap vanishes and the persistence
function satisfies  $\int_0^\infty dt \,
F(\sqrt{t})=\infty$ (see Theorem 6.17 and Corollary 6.18 in \cite{noi}).
\end{remark}
\subsection{Some ideas of the strategy for proving theorems \ref{main
    theorem 01}, \ref{persf}}
\label{news} 
The main idea behind the proof of theorem \ref{main
    theorem 01} goes as follows. First of all one covers the lattice
  with non overlapping 
  cubic blocks $\{\Lambda_{\ell x}\}_{x\in {\mathbb Z}^d}$ and, on the rescaled
  lattice ${\mathbb Z}^d(\ell):=\ell {\mathbb Z}^d$, one considers the new model with
  single spin space $S=\{0,1\}^{\ell^d}$, good event $G:=G_\ell$, single
  site measure the restriction of $\mu$ to $S$ and
  renormalized constraints $\{c^{ren}_x\}_{x\in {\mathbb Z}^d(\ell)}$ which are a strengthening of the North-East ones namely
  \begin{equation*}
    c^{ren}_{x}(\eta)=1 \quad \text{ iff } \eta_y\in G \text{ for all } y\in {\mathcal K}_x^*.
  \end{equation*}
Such a model is referred to in \cite{noi} as the \emph{*-general
  model}. By assumption the probability of $G$ can be made arbitrarily
close to one by taking $\ell$ large enough and therefore, by the so
called \emph{Bisection-Constrained approach} which is detailed in the
next section for the case when $\mu$ is a high
temperature Gibbs measure, the spectral gap of the
*-general model is positive. Next one observes that  assumption (b)
of the theorem is there exactly to allow one to reconstruct any \emph{legal} move of
the *-general model, {\it i.e. \ } a full update of an entire block of spins, by means of a finite (depending only on $\ell$)
sequence of \emph{legal} moves for the original 0-1 KCMS. It is then an
easy step, using standard path techniques for comparing two different
Markov chains (see e.g. \cite{Saloff}), to go from the Poincar\'e
inequality for the *-general model to the Poincar\'e inequality for the
original model.  

The proof of (a slightly less precise version of) Theorem \ref{persf} given in \cite{noi} is based on the
Feynman-Kac formula and standard
large deviation considerations. However it is possible to provide a
simpler and more precise argument as follows. One first observe that $F(t)=F_1(t)+F_0(t)$ where
$$
F_1(t)= \int\, d\mu(\eta)\,
{\mathbb P}(\sigma^\eta_0(s)=1\ \text{for all $s\le t$})
$$
and similarly for $F_0(t)$. Consider now  
$F_1(t)$, the case of $F_0(t)$ being similar, and define $T_A(\eta)$ as the
hitting time of the set $A:=\{\eta :\ \eta_0=0\}$ starting from the
configuration $\eta$. Then (see e.g. Theorem 2 in \cite{Asselah-DaiPra})
$$
F_1(t)={\mathbb P}_\mu\Bigl(T_A>t\Bigr)\leq e^{-t \lambda_A}
$$
where ${\mathbb P}_\mu$ denotes the probability over the process started from
the equilibrium distribution $\mu$ and $\lambda_A$ is given by the
variational formula for the Dirichelt problem
\begin{equation}
  \label{eq:Dirichlet eigenvalue}
  \lambda_A := \inf\Bigl\{{\cal D}(f):\ \mu(f^2)=1,\ f\equiv 0 \text{ on
    } A\Bigr\}
\end{equation}
Notice that for any $f$ as above $\mathop{\rm Var}\nolimits(f)\geq \mu(A)=q$. Therefore 
$\lambda_A\geq q\mathop{\rm gap}\nolimits$ and the proof is complete. 
\subsection{Asymptotics of the spectral gap near the ergodicity threshold.}
An important question, particularly in connection with numerical
simulations or non-rigorous approaches, is the behavior near the
ergodicity threshold $q_c$ of the spectral gap for each specific
model. Here is a set of results proven in \cite{noi}.
  \begin{enumerate}[]    
\item  {\bf East Model.} 
\begin{equation}
\lim_{q\to 0} \log(1/\mathop{\rm gap}\nolimits)/(\log(1/q))^2 =\left(2\log 2\right)^{-1}
\label{eq:th2}
\end{equation}
\item {\bf FA-1f.} For any $d\ge 1$, there exists a constant $C=C(d)$
such that for any $q \in (0,1)$, the spectral gap on
${\mathbb Z}^d$ satisfies:
$$
\begin{array}{rcccll}
\displaystyle
C^{-1} q^3
& \leq &
\displaystyle \mathop{\rm gap}\nolimits({\mathcal L})
& \leq &
C q^3
& \qquad \text{for } d=1, \cr
\displaystyle
C^{-1} q^2/\log(1/q) &
\leq &
\displaystyle \mathop{\rm gap}\nolimits({\mathcal L})
& \leq &
\displaystyle C q^2
& \qquad \text{for } d=2, \cr
\displaystyle
C^{-1} q^2 &
\leq &
\displaystyle \mathop{\rm gap}\nolimits({\mathcal L}) & \leq
& C q^{1+\frac{2}{d}}
& \qquad \text{for } d \geq 3 .
\end{array}
$$
\item {\bf FA-df in ${\mathbb Z}^d$.}
Fix $\epsilon >0$. Then there exists $c=c(d)$ such that
 \begin{align}
   \label{eq:FA.1}
\Bigl[\exp^{d-1}(c/q^2)\Bigr]^{-1} \le& \mathop{\rm gap}\nolimits({\mathcal L}) \le
\Bigl[\exp^{d-1}\bigl(\frac{\lambda_1-\epsilon}{q}\bigr)\Bigr]^{-1} \quad
&d\ge 3 \cr
\exp(-c/q^5) \le& \mathop{\rm gap}\nolimits({\mathcal L}) \le
\exp\bigl(-\frac{(\lambda_1-\epsilon)}{q}\bigr)\quad
&d=2
 \end{align}
as $q\downarrow 0$, where $\exp^{d-1}$ denote the $(d-1)^{\rm th}$-iterate of the exponential
function and $\lambda_1=\pi^2/18$.
  \end{enumerate}
The proof of the lower bounds is a rather delicate combination of the
renormalization scheme described above together with paths
techniques as described in \cite{Saloff}. The upper bounds are proved instead either by a careful
choice of a test function in the variational characterization of the
spectral gap or by a lower bound on the persistence function $F(t)$
combined with the upper bound given in Theorem \ref{persf}.  
\begin{remark}
Again some of the above findings disprove previous claims for the East
model \cite{SE} and for the FA-1f
model in $d=2,3$ \cite{WBG1}. The result for the East model actually
came out as
a surprise. In \cite{SE} the model was considered ``essentially'' solved and the
result for the spectral gap was $\mathop{\rm gap}\nolimits\approx q^{\log_2(q)}$ as
$q\downarrow 0$ to be compared to the correct scaling $q^{\log_2(q)/2}$. In \cite{Aldous} the above solution was proved to be a
\emph{lower} bound and an upper bound of the form 
$q^{\log_2(q)/2}$ was rigorously established but considered poor 
because 
off by a power $1/2$ from the supposedly correct behavior. 

The scaling indicated in \cite{SE} is based in part on the following
consideration. Fix $q\ll 1$ and consider the East model on the interval
$\Lambda_q:=[0,\dots,1/q]$ with the last site free to flip ({\it i.e. \ } zero boundary
conditions). Notice that $1/q$ is the average distance between the
zeros. Start from the configuration identically equal to one and let $T$
be the (random) time at which the origin is able to flip. Energy
barriers consideration (see \cite{Aldous2,Aldous,Chung}) suggest that
${\mathbb E}(T)$ should scale as $q^{\log_2(q)}$ and that is what
was assumed in \cite{SE}.  However it is not difficult to prove that the
scaling of ${\mathbb E}(T)$ is bounded above by $(q \mathop{\rm gap}\nolimits)^{-1}$.  Indeed
we can write for any $t\ge 0$
$$
\exp(-c q\mathop{\rm gap}\nolimits({\mathcal L}_{\Lambda_q}) t)\ge \tilde F(t)\ge \mu(\text{all ones}){\mathbb P}(T\ge t)
\geq e^{-2} {\mathbb P}(T\ge t)
$$
where $\tilde F(t)$ is the finite volume persistence
function. Integrating over $t$ and using the monoticity of the gap (see \cite[Lemma 2.11]{noi}) give 
${\mathbb E}(T)\le e^{2}c (q\mathop{\rm gap}\nolimits({\mathcal L}_{\Lambda_q})))^{-1}
\le e^{2}c (q\mathop{\rm gap}\nolimits({\mathcal L}))^{-1}$. This,
in view of Theorem \ref{main theorem 01}, is incompatible
with the assumed scaling $q^{\log_2(q)}$. 

Moreover one can obtain a lower bound on ${\mathbb E}(T)$ as follows.
Let $\lambda$ be such that ${\mathbb P}(T\ge \lambda)= e^{-1}$ then
clearly ${\mathbb P}(T\ge t)\le e^{-\lfloor{t/\lambda}\rfloor}$ and ${\mathbb E}(T)\ge
e^{-1}\lambda$. We can always couple in the natural way two copies of the process,
one  started from all
ones and the other from any other configuration $\eta$, and conclude that 
$$
{\mathbb P}(\text{the two
  copies have not coupled at time } t)\le {\mathbb P}(T\ge t)\le e^{1-\lambda t}.
$$ 
Standard arguments give immediately that $\mathop{\rm gap}\nolimits^{-1}\le \lambda$ {\it i.e. \ }
${\mathbb E}(T)\ge e^{-1} \mathop{\rm gap}\nolimits^{-1}$. In conclusion
$$
e^{-1} \bigl[\mathop{\rm gap}\nolimits({\mathcal
  L}_{\Lambda_q})\bigr]^{-1}\le {\mathbb E}(T) \le 
e^{2}c \bigl[q\mathop{\rm gap}\nolimits({\mathcal L}_{\Lambda_q})\bigr]^{-1}
$$

\end{remark}
\section{Extension to interacting models} 
In this section we show how to extend the results on the positivity of
the spectral  gap for 0-1 KCSM
on a regular lattice ${\mathbb Z}^d$ to the case in which a weak interaction is
present among the spins. 
We begin by defining what we mean by an \emph{interaction}.
\begin{definition}
A finite range interaction $\Phi$ is a collection $\Phi:=\{\Phi_\Lambda\}_{\Lambda\in {\mathbb F}}$ where 
\begin{enumerate}[i)]
\item $\Phi_\Lambda:\Omega_\Lambda\mapsto {\mathbb R}$ for every $\Lambda\in {\mathbb F}$;
\item $\Phi_\Lambda=0$ if ${\rm diam}(\Lambda)\ge r$ for some finite $r=r(\Phi)$ called
  the range of the interaction;
\item $\|\Phi\|\equiv \sup_{x\in {\mathbb Z}^d}\sum_{\Lambda\ni x}\ninf{\Phi_\Lambda}<\infty$; 
\end{enumerate}
We will say that $\Phi\in {\mathcal B}_{M,r}$ if $r(\Phi)\le r$ and
$\|\Phi\|\le M$. 
\end{definition}
Given an interaction $\Phi\in {\mathcal B}_{r,M}$ and $\Lambda\in {\mathbb F}$, we define the energy in
$\Lambda$ of a spin configuration $\sigma\in \Omega$ by
\begin{equation*}
  H_\Lambda(\sigma) = \sum_{A\cap \Lambda\neq \emptyset}\Phi_A(\sigma) 
\end{equation*}
For $\sigma\in \Omega_\Lambda$ and $\tau\in \Omega_{\Lambda^c}$ we also let
$H_\Lambda^\tau(\sigma):=H_\Lambda(\sigma\cdot\tau)$ where $\sigma\cdot\tau$ denotes the
configuration equal to $\sigma$ inside $\Lambda$ and to $\tau$ outside it. Finally,
for any $\Lambda\in{\mathbb F}$ and $\tau\in \Omega_{\Lambda^c}$ , we define the finite volume
Gibbs measure on $\Omega_\Lambda$ with boundary conditions $\tau$ and apriori
single spin measure $\nu$ by the formula
\begin{equation*}
  \mu_\Lambda^{\Phi,\tau}(\sigma):= \frac{1}{Z_\Lambda^{\Phi,\tau}} e^{-
    H_\Lambda^\tau(\sigma)}\prod_{x\in \Lambda}\nu(\sigma_x) 
\end{equation*}
where $Z_\Lambda^{\Phi,\tau}$ is a normalization constant.

The key property of Gibbs measures is that, for any $V\subset \Lambda$ and any
$\xi$ in $\Lambda\setminus V$, the conditional Gibbs measure in $\Lambda$ with
boundary conditions $\tau$ given $\xi$ 
coincides with the Gibbs measure in $V$ with boundary condition $\tau_{\Lambda^c}\cdot \xi$. More formally
\begin{equation*}
  \mu_\Lambda^{\Phi,\tau}(\cdot
\thinspace | \thinspace\sigma_{V^c}=\xi) = \mu_V^{\Phi,{\tau_{\Lambda^c}\cdot \xi}}(\cdot)
\end{equation*}
Clearly averages w.r.t. $ \mu_\Lambda^{\Phi,\tau}(\cdot \thinspace | \thinspace\sigma_{V^c}=\xi)$ are
function of $\xi$ and, whenever confusion does not arise, we will
systematically drop $\xi$ from our notation.

As it is well known (see e.g. \cite{Simon}), for any $r<\infty$ there exists
$M_0>0$ such that for any $0<M< M_0$ the following holds. For any $\Phi\in {\mathcal B}_{r,M}$ 
there exists a unique probability measure $\mu^{\Phi}$ on $\Omega$, called the unique Gibbs measure associated to
the interaction $\Phi$ with apriori measure $\nu$, such that, for
any $\tau$,  
\begin{equation*}
\lim_{\Lambda\uparrow {\mathbb Z}^d} \mu_\Lambda^{\Phi,\tau}=\mu^\Phi 
\end{equation*}
where the limit is to be understood as a weak limit. 
Moreover the limit is reached ``exponentially fast'' in the strongest
possible sense. Namely, for any $\D\subset \Lambda\in {\mathbb F}$ and any two boundary conditions
$\tau,\tau'$, 
\begin{equation}
  \max_{\sigma_\D}\Big|\frac{\mu_\Lambda^{\Phi,\tau'}(\sigma_\D)}{\mu_\Lambda^{\Phi,\tau}(\sigma_\D)}-1  \Big|\le
  K |D_r(\tau,\tau')|\,e^{-md\bigl(\D,\,D_r(\tau,\tau')\bigr)} 
\label{strong mixing}
\end{equation}
where $D_r(\tau,\tau')=\{y:\ 0<d(y,\Lambda)\le r,\ \tau_y\neq \tau'_y\}$ and the
constants $m,K$
depend only on $M,r,d$. Moreover $m\uparrow +\infty$ as $M\downarrow 0$.
When $d=1$ the
threshold $M_0$ can be taken equal to $+\infty$. 
In all what follows we will always assume that $\Phi\in {\mathcal B}_{r,M}$ for
some apriori given $r,M$ and that $M<M_0$.
\begin{remark}
In general the constant $M_0$ does not coincide with any ``critical
point'' for the model. It is only a sort of ``high temperature
threshold'' (see \cite{MO1,MO2,MOS} for more details about this issue).   
\end{remark}
Having described the notion of the unique Gibbs measure corresponding to
$\Phi$, we can define the generator ${\mathcal L}^{\Phi}$ of a $0$-$1$ KCSM with
interaction $\Phi$ and constraints $c_x$ given by \eqref{constraints},
as the unique self-adjoint operator on $L^2(\Omega,\mu^\Phi)$ with quadratic
form 
\begin{equation*}
  {\mathcal D}^{\Phi}(f)=\sum_x\mu^{\Phi}\left(c_x \mathop{\rm Var}\nolimits^\Phi_x(f)\right),\quad \text{ $f$ local}
\end{equation*}
where now the local variance $\mathop{\rm Var}\nolimits_x^{\Phi}(f)$ is computed with the
conditional Gibbs measure given all the spins outside $x$. The construction of
the generator in a finite volume $\Lambda$ with boundary
conditions $\tau$ is exactly the
same as in the non-interacting case and we skip it. 

\subsection{Spectral gap for a weakly interacting North-East model}
Instead of trying to prove a very general result on the spectral gap of
a weakly interacting KCSM, we will explain how to deal with the
interaction in the concrete case of the North-East model introduced in
section \ref{models}. Moreover, in order not to obscure the discussion with
renormalization or block constructions, we will make the unnecessary
assumption that the basic parameter $q$ of the reference measure $\nu$
is very close to one.

\begin{Theorem}
\label{interacting NE}
Let $\{c_x\}_{ x\in {\mathbb Z}^2}$ be those of the North-East model. There
exists $q_0\in (0,1)$ and for any $r<\infty$ there exists $M_1$ such that, for any
$M<\min(M_0,M_1)$ and $q\ge q_0$,
$$
\inf_{\Phi\in {\mathcal B}_{r,M}}\mathop{\rm gap}\nolimits({\mathcal L}^{\Phi}) >0
$$   
\end{Theorem}
\begin{remark}
  As we will see in the proof of the theorem, the restriction on
  strength of the interaction comes from two different requirements. The
  first one is that the finite volume Gibbs measure has the very strong
  mixing property uniformly in the boundary conditions given in
  \eqref{strong mixing}. That, as we
  pointed out previously, is guaranteed as long as $M<M_0$. The second
  one requires that the zeros, which certainly percolate in a robust way
  w.r.t. the unperturbed measure $\nu$ because of the assumption
  $q\approx 1$, continue to do so even when we switch on the
  interaction. It is worthwhile to observe that for the one dimensional 
  East model, the first requirement is satisfied for any $M<\infty$ and
  that the second one is simply not necessary. Therefore for the East
  model  the above theorem should be reformulated as follows.
  \begin{Theorem}
Let $\{c_x\}_{ x\in {\mathbb Z}}$ be those of the East model. For any finite
pair $(r,M)$
$$
\inf_{\Phi\in {\mathcal B}_{r,M}}\mathop{\rm gap}\nolimits({\mathcal L}^{\Phi}) >0
$$   
\end{Theorem}
\end{remark}
\begin{proof}[of Theorem \ref{interacting NE}]
We will follow the pattern of the proof for the non interacting case
given in \cite{noi} and we will establish the stronger result
\begin{equation}
  \sup_{\Lambda\in{\mathbb F}} \gamma(\Lambda) <+\infty,\quad \text{where}\quad 
\gamma(\Lambda):=\Bigl(\inf_{\Phi\in {\mathcal B}_{r,M}}\inf_{\tau\in {\rm Max}_\Lambda}\mathop{\rm gap}\nolimits({\mathcal L}_\Lambda^{\Phi,\tau})\Bigr)^{-1}
\label{stronger}
\end{equation}
provided that $q>q_0$ is large and $M$ is taken sufficiently
small. Above ${\rm Max}_\Lambda$ denotes the set of configurations in
$\Omega_{\Lambda^c}$ which are identically equal to zero on $\partial_+^*\Lambda$. In
what follows in order to simplify the notation we will not write the
dependence on the boundary condition of the transition rates.

As in \cite{noi}  the first step consists in proving a certain
monotonicity property of $\gamma(\Lambda)$.

\begin{Lemma}
\label{Mon}
For any $V\subset \Lambda\in {\mathbb F}$,
$$
0<\gamma(V)\le \gamma(\Lambda)<\infty
$$
\end{Lemma}
\begin{proof}[Proof of the Lemma]
Fix $\Phi\in {\mathcal B}_{r,M}$ and, for any $\xi\in {\rm Max}_V$, define the new interaction
$\Phi^\xi$ as follows:
\begin{equation*}
  \Phi^\xi_A(\sigma_A)=
  \begin{cases}
   0 & \text{ if $A\cap V^c\neq\emptyset$}\cr
 \sum_{A':\, A'\cap V=A}\Phi_{A'}(\sigma_A\cdot \xi_{A'\setminus A}) &
 \text{ if $A\subset V$}
  \end{cases}
\end{equation*}
Notice that, by construction, 
$$
r(\Phi^\xi)\le r(\Phi)\quad \text{and}\quad \sup_x\sum_{A\ni x}\ninf{\Phi_A^\xi}\le \ninf{\Phi}
$$ 
so that $\Phi^\xi\in {\mathcal B}_{r,M}$. Next observe that the Gibbs
measure on $\Lambda$ with interaction $\Phi^\xi$ is simply the product measure 
$$ 
\mu^{\Phi^\xi}_{\Lambda}(\sigma_\Lambda):=\mu^{\Phi,\xi}_V(\sigma_V)\otimes\nu_{\Lambda\setminus
  V}(\sigma_{\Lambda\setminus V})\quad \text{on}\quad \Omega_\Lambda=\Omega_V\otimes\Omega_{\Lambda\setminus V}
$$
Thus, for any $f\in L^2(\Omega_V,\mu_V^{\Phi,\xi})$ and $\tau\in\rm Max_{\Lambda}$,
we can write ($\mathop{\rm Var}\nolimits_\Lambda^{\Phi,\tau}\equiv \mathop{\rm Var}\nolimits_{\mu_\Lambda^{\Phi,\tau}}$)
\begin{gather*}
\mathop{\rm Var}\nolimits_V^{\Phi,\xi}(f)=\mathop{\rm Var}\nolimits_\Lambda^{\Phi^\xi,{\tau}}(f) \hfill\break
\le\,\gamma(\Lambda)\,{\mathcal D}_\Lambda^{\Phi^\xi,\tau}(f)
\hfill\break 
\le \gamma(\Lambda)\,{\mathcal D}_V^{\Phi,\xi}(f)
\end{gather*}
where, in the last inequality, we used the fact that, for any $x\in V$ and any $\omega\in\Omega_\Lambda$,
$c_{x,\Lambda}(\omega)\le c_{x,V}(\omega)$ because $\xi\in {\rm Max}_V$, together with 
$$
\mathop{\rm Var}\nolimits_\Lambda^{\Phi^\xi,{\tau}}(f \thinspace | \thinspace\{\sigma_y\}_{y\neq x}) = \mathop{\rm Var}\nolimits_V^{\Phi,\xi}(f \thinspace | \thinspace\{\sigma_y\}_{y\neq x}).
$$ 
\qed \end{proof}
\noindent
Thanks to Lemma \ref{Mon} we need to prove \eqref{stronger} only when
$\Lambda$ runs through all possible rectangles. For this purpose our main
ingredient will be the bisection technique of \cite{SFlour} which, in its
essence, consists in proving a suitable recursion relation between
spectral gap on scale $2L$ with that on scale $L$, combined with the
novel idea of considering an accelerated block dynamics which is itself
constrained. Such an approach is referred to in \cite{noi} as the
\emph{Bisection-Constrained} or \emph{B-C approach}.

In order to present it we first need to recall some simple facts from two
dimensional percolation. 

A \emph{path} is a collection $\{x_0,x_1,\dots,x_n\}$ of distinct points
in ${\mathbb Z}^2$ such that $d(x_i,x_{i+1})=1$ for all $i$. A \emph{$*$-path}
is a collection $\{x_0,x_1,\dots,x_n\}$ of distinct points in ${\mathbb Z}^2$
such that $x_{i+1}\in {\mathcal N}_{x_i}^*$ for all $i$. Given a rectangle $\Lambda$
and a direction $\vec e_i$, we will say that a path $\{x_0,\dots,x_n\}$
traverses $\Lambda$ in the $i^{th}$-direction if $\{x_0,\dots,x_n\}\subset \Lambda$
and $x_0,x_n$ lay on the two opposite sides of $\Lambda$ orthogonal to $\vec
e_i$.
   
\begin{definition}
  Given a rectangle $\Lambda$ and a configuration $\omega\in\Omega_\Lambda$,
  a path $\{x_0,\dots,x_n\}$ is called a \emph{top-bottom crossing}
  (\emph{left-right crossing}) if it traverses $\Lambda$ in the
   vertical (horizontal) direction and $\omega_{x_i}=0$ for all
  $i=0,\dots,n$. The  rightmost (lower-most) such crossings (see \cite{Grimmett} page 317) 
will be denoted by $\Pi_\omega$ 
\end{definition}

\begin{remark}\label{rem1}
Given a rectangle $\Lambda$ and a path $\Gamma$ traversing $\Lambda$
in e.g. the vertical direction, let $\Lambda_\Gamma$
consists of all the sites in $\Lambda$ which are in $\Gamma$ or to the right of
it. Then, as remarked in \cite{Grimmett}, the event
$\{\omega:\ \Pi_\omega=\Gamma\}$ depends only on the variables $\omega_x$ with
$x\in \Lambda_\Gamma$. 
\end{remark}

We are now ready to start the actual proof of the theorem. 
At the beginning the method requires a simple geometric result (see
\cite{Cesi}) which we now describe.

Let $l_k := (3/2)^{k/2}$, and let ${\mathbb F}_k$ be the set of all rectangles
$\Lambda\subset {\mathbb Z}^2$
which, modulo translations and permutations of the coordinates, are
contained in
$
  [ 0,l_{k+1} ] \times [0,l_{k+2}]
$.
The main property of ${\mathbb F}_k$ is that each rectangle in 
${\mathbb F}_k\setminus {\mathbb F}_{k-1}$ can be obtained as a ``slightly
overlapping union'' of two rectangles in ${\mathbb F}_{k-1}$. 

\begin{Lemma}
\label{geom}
For all $k\in {\mathbb Z}_+$,
for all $\Lambda \in {\mathbb F}_k\setminus {\mathbb F}_{k-1}$ there exists a finite sequence
$\{\Lambda_1^{(i)}, \Lambda_2^{(i)}\}_{i=1}^{s_k}$ in ${\mathbb F}_{k-1}$, where 
$s_k := \lfloor l_k^{1/3} \rfloor$, such that, letting $\delta_k := \frac 18
\sqrt{l_k}-2$,  
\begin{enumerate}[(i)]
\item $\Lambda = \Lambda_1^{(i)} \cup \Lambda_2^{(i)}$, 
\item  $d(\Lambda\setminus \Lambda_1^{(i)}, \Lambda\setminus \Lambda_2^{(i)}) \ge \delta_k $,  
\item $\left(\Lambda_1^{(i)}\cap \Lambda_2^{(i)}\right)\cap \left(\Lambda_1^{(j)}\cap \Lambda_2^{(j)}\right) 
        = \emptyset$,  if $i\ne j$.   
\end{enumerate}
\end{Lemma}
The \emph{B-C approach} then establishes a simple recursive inequality
between the quantity $\gamma_k := \sup_{\Lambda\in{\mathbb F}_k} \gamma(\Lambda)$ on
scale $k$ and the same quantity on scale $k-1$ as follows.

Fix $\Lambda\in {\mathbb F}_k\setminus {\mathbb F}_{k-1}$ and
write it as $\Lambda = \Lambda_1 \cup \Lambda_2$ with $\Lambda_1, \Lambda_2\in {\mathbb F}_{k-1}$
satisfying the properties described in Lemma \ref{geom} above. Without
loss of generality we can assume that all the horizontal faces of $\Lambda_1$ and of
$\Lambda_2$ lay on the horizontal faces of $\Lambda$ except for the face orthogonal to the
first direction $\vec e_1$ and that, along that direction, $\Lambda_1$ comes
before $\Lambda_2$.
Set $\D\equiv \Lambda_1 \cap \Lambda_2$ and write, for definiteness,
$\D=[a_1,b_1]\times[a_2,b_2]$. Lemma \ref{geom} implies that
the width of $\D$ in the first direction, $b_1-a_1$, is at least
$\delta_k$.
Set also 
$$
I\equiv[a_1 + (b_1-a_1)/2,\ b_1]\times[a_2,b_2]
$$ 
and let
$\partial_{r}I=\{b_1\}\times[a_2,b_2]$
be the right face of $I$ along the first direction.
\begin{definition}
Given a configuration $\omega\in \Omega$ we will say that $\omega$ is $I$-good iff
there exists a top-bottom crossing of $I$.   
\end{definition}
Given $\tau\in {\rm Max}_\Lambda$, we run the following constrained ``block dynamics''
on $\Omega_\Lambda$ (in what follows, for simplicity, we suppress the index $i$)
with boundary conditions $\tau$ and blocks $B_1:=\Lambda_1\setminus I$, $B_2:=\Lambda_2$.  The block $B_2$ waits a mean one
exponential random time and then the current configuration inside it is
refreshed with a new one sampled from the Gibbs measure of the block
given the previous configuration outside it (and $\tau$ outside $\Lambda$). The block $ B_1$ does
the same but now the configuration is refreshed only if the current
configuration $\omega$ in $B$ is $I$-good (see Figure \ref{blocchi}).

\begin{figure}[ht]
\label{blocchi}
\psfrag{l1}{$B_1$}
\psfrag{l2}{$\!\!\!\!\!\!\! B_2=\Lambda_2$}
\psfrag{d}{$\Delta$}
\begin{center}
\includegraphics[width=.5\columnwidth]{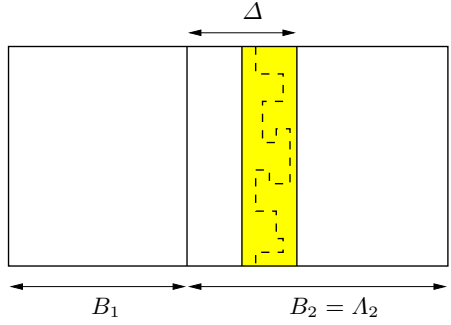}
\caption{
The rectangle $\Lambda$ divided into two blocks $B_1$ and $B_2$.
The grey region is the strip $I$ with a top-bottom crossing.}
\end{center}
\end{figure}

The generator of the
block dynamics applied to $f$ can be written as
\begin{equation}
\label{gen}
{\mathcal L}_{\rm block} f=c_1(\mu^{\Phi,\tau}_{ B_1}(f)-f)+\mu^{\Phi,\tau}_{B_2}(f)-f
\end{equation}
and the associated Dirichlet form is
$$
{\mathcal D}^{\Phi,\tau}_{\rm block}(f)=\mu^{\Phi,\tau}_\Lambda\left(c_1\mathop{\rm Var}\nolimits^{\Phi}_{ B_1}(f)+\mathop{\rm Var}\nolimits^{\Phi}_{B_2}(f)\right)
$$
where $c_1(\omega)$ is just the indicator of the event 
that $\omega$ is $I$-good.
\begin{remark}
  The reader should keep in mind that the e.g. the notation
  $$\mu^{\Phi,\tau}_\Lambda\bigl(c_1\mathop{\rm Var}\nolimits^{\Phi}_{ B_1}(f)\bigr)$$ stands for
  $\sum_{\xi}\mu^{\Phi,\tau}_\Lambda(\xi)c_1(\xi)\mathop{\rm Var}\nolimits^{\Phi,\xi}_{
    B_1}(f)$ and that one can imagine the sum restricted to those
  configurations outside $B_1$ that coicide with $\tau$ outside $\Lambda$ since
  otherwise their probability $\mu^{\Phi,\tau}_\Lambda(\xi)$ is zero.
\end{remark}
In order to study the mixing property of the chain we need
the following two lemmas.

\begin{Lemma}[\cite{SFlour}]
\label{SM}
Fix $(r,M)$ with $M<M_0$. Then, for any 
$\Phi\in {\mathcal B}_{r,M}$, 
\begin{eqnarray}
\sup_{\tau'}|\mu_{ B_1}^{\Phi,\tau'}(g)-\mu_\Lambda^{\Phi,\tau}(g)| & \le\,\lambda_k\,\ninf{g}\quad\forall
g:\Omega_{B_2^c}\mapsto {\mathbb R} \label{inA}\cr
\sup_{\tau'}|\mu_{B_2}^{\Phi,\tau'}(g)-\mu_\Lambda^{\Phi,\tau}(g)| & \le\,\lambda_k\,\ninf{g}\quad\forall
g:\Omega_{ B_1^c}\mapsto {\mathbb R}
\label{inB}
\end{eqnarray}
where $\lambda_k:= Kr l_{k+1}e^{-m\delta_k/2}$ and the constants $K,m$ are given in \eqref{strong mixing}. 
\end{Lemma}

\begin{Lemma}
\label{pippo}
There
exists $q_0\in (0,1)$ and for any $r<\infty$ there exists $M_1$ such that, for any
$M<\min(M_0,M_1)$ and $q\ge q_0$,
\begin{equation*}
\varepsilon_k:= \max_{\Phi\in {\mathcal B}_{r,M}}\max_{\tau}\,\mu_{B_2}^{\Phi,\tau}(\omega\text{ is not
  $I$-good }) \le e^{-\delta_k}.
\end{equation*}
\end{Lemma}
\begin{proof}
It follows immediately from standard percolation arguments together
with 
$$
\sup_{\Phi\in {\mathcal B}_{r,M}}\sup_\tau\mu_{\{x\}}^{\Phi,\tau}(\sigma_x=1)\le (1-q)e^{2M}
$$
\qed \end{proof}
We can now state the main consequence of Lemma \ref{SM}, \ref{pippo}.
\begin{Proposition}
\label{gapblock}
There exists $q_0\in (0,1)$ and for any $r<\infty$ there exists $M_1$
such that, for any $M<\min(M_0,M_1)$ and $q\ge q_0$,
\begin{equation}
  \label{eq:inter5}
\gamma^{(k)}_{{\rm block}}:=\sup_{\Phi\in {\mathcal B}_{r,M}}\sup_{\tau\in {\rm
    Max}_\Lambda}\left(\mathop{\rm gap}\nolimits({\mathcal L}^{\Phi,\tau}_{\rm block})\right)^{-1} 
\le \left(1-8\sqrt{2\lambda_k+\varepsilon_k}\right)^{-1} 
\end{equation}
for all $k$ so large that the r.h.s. of \eqref{eq:inter5} is smaller than $2$.
\end{Proposition}

\begin{proof}[Proof of the proposition]
Fix $r,M$ and $\Phi$ as prescribed, let $\tau\in {\rm Max}_\Lambda$ and, in order to simplify the
notation, drop all the superscripts $\Phi,\tau$. 
Let $f:\Omega_\Lambda\mapsto {\mathbb R}$ be a mean zero function, the  
 eigenvalue equation associated to the generator (\ref{gen}) is
\begin{equation}\label{eigen}
c_1(\mu_{ B_1}(f)-f)+\mu_{B_2}(f)-f=\lambda f
\end{equation}
By construction $\lambda\ge -2$. 

Assume that $\lambda>-1+\sqrt{\lambda_k}$ since otherwise there is nothing to be
proved.  
By applying $\mu_{ B_1}$ to both
sides of (\ref{eigen}) and using \eqref{inA} we obtain
\begin{equation}\label{muA}
(1+\lambda)\mu_{ B_1}f=\mu_{ B_1}\bigl(\mu_{B_2}(f)\bigr) \quad
\Rightarrow \quad  \ninf{\mu_{ B_1}(f)}  \le \sqrt{\lambda_k}\,\ninf{\mu_{B_2}(f)}
\end{equation}
If we rewrite \eqref{eigen} as 
\begin{equation*}
  f= \frac{1}{1+\lambda+c_1}\mu_{B_2}(f) + \frac{c_1}{1+\lambda+c_1}\mu_{ B_1}(f)
\end{equation*}
and apply $\mu_{B_2}$ to both sides, by using (\ref{muA}) together with the assumption $\lambda>-1+\sqrt{\lambda_k}$, we get 
\begin{gather}
\label{muB}
\ninf{\mu_{B_2}(f)} \le \ninf{\mu_{B_2}(f)}\,
\ninf{\mu_{B_2}(\frac{1}{1+\lambda+c_1})} \cr
+ \lambda_k\ninf{\frac{c_1}{1+\lambda+c_1}}\ninf{\mu_{ B_1}(f)}\cr
\le \ninf{\mu_{B_2}(f)}\left( \ninf{\mu_{B_2}(\frac{1}{1+\lambda+c_1})}  +
\sqrt{\lambda_k} \right)
\end{gather}
which is possible only if 
$$
\ninf{\mu_{B_2}(\frac{1}{1+\lambda+c_1})} \ge 1-\sqrt{\lambda_k}
$$
{\it i.e. \ } 
$$
\lambda \le -1+8\sqrt{2\lambda_k+\epsilon_k}
$$
and the proof is complete.  
\qed \end{proof}
By writing down the standard Poincar\'e inequality for the
block auxiliary chain, we get that for any $f$
\begin{equation}
  \label{eq:s1}
  \mathop{\rm Var}\nolimits^{\Phi,\tau}_\Lambda(f)\le \gamma^{(k)}_{{\rm block}}\ \mu^{\Phi,\tau}_\Lambda\Bigl(c_1\mathop{\rm Var}\nolimits^{\Phi}_{ B_1}(f)+\mathop{\rm Var}\nolimits^{\Phi}_{B_2}(f)\Bigr)
\end{equation}
The second term in the r.h.s. of \eqref{eq:s1}, using the definition of $\gamma_k$ and the fact that
$B_2=\Lambda_2\in {\mathbb F}_{k-1}$ is bounded from above
by 
\begin{equation}
  \label{eq:s2}
  \mu^{\Phi,\tau}_\Lambda\Bigl(\mathop{\rm Var}\nolimits^{\Phi}_{B_2}(f)\Bigr)\le
  \gamma_{k-1}\sum_{x\in B_2}\mu^{\Phi,\tau}_\Lambda\bigl(c_{x,B_2} \mathop{\rm Var}\nolimits^{\Phi}_x(f)\bigr)
\end{equation}
Notice that, by construction, for all $x\in B_2$ and all $\omega$,
$c_{x,B_2}(\omega)=c_{x,\Lambda}(\omega)$.  Therefore the term
$\sum_{x\in B_2}\mu^{\Phi,\tau}_\Lambda\bigl(c_{x,B_2}
\mathop{\rm Var}\nolimits^{\Phi}_x(f)\bigr)$ is nothing but the
contribution carried by the set $B_2$ to the full Dirichlet form ${\mathcal D}^{\Phi,\tau}_\Lambda(f)$. 

Next we examine the more complicate term
$\mu^{\Phi,\tau}_\Lambda\Bigl(c_1\mathop{\rm Var}\nolimits^{\Phi}_{ B_1}(f)\Bigr)$. For any $\omega$
such that there exists a rightmost crossing $\Pi_\omega$ in $I$ denote by
$\Lambda_{\omega}$ the set of all sites in $\Lambda$ which are to the \emph{left} of $\Pi_\omega$.
Since
$\mathop{\rm Var}\nolimits^{\Phi}_{ B_1}(f)$ depends only on $\omega_{\Lambda\setminus  B_1}$ and,
for any top-bottom crossing $\Gamma$ of $I$,
${1 \mskip -5mu {\rm I}}_{\{\Pi_\omega=\Gamma\}}$ does not depend on the variables $\omega$'s to the
left of $\Gamma$, we can write
\begin{gather}
  \mu^{\Phi,\tau}_\Lambda\Bigl(c_1\mathop{\rm Var}\nolimits^{\Phi}_{{ B_1}}(f)\Bigr)=
  \mu^{\Phi,\tau}_\Lambda\Bigl({1 \mskip -5mu {\rm I}}_{\{\exists\, \Pi_\omega \text{ in }I\}}  \mu^{\Phi}_{\Lambda_\omega}\bigl(\mathop{\rm Var}\nolimits^{\Phi}_{{ B_1}}(f)\bigr)\Bigr)
\label{A}
\end{gather}
The convexity of the variance implies that 
\begin{equation*}
\mu^{\Phi}_{\Lambda_\omega}\bigl(\mathop{\rm Var}\nolimits^{\Phi}_{{ B_1}}(f)\bigr) \le 
\mathop{\rm Var}\nolimits^{\Phi}_{{\Lambda_\omega}}(f)
\end{equation*}
where it is understood that the r.h.s. depends on the variables in
$\Pi_\omega$ and to the right of it. The key observation at this stage,
which explains the role and the need of the event $\{\exists\ \Pi_\omega
\text{ in }I\}$,
is the following. For any $\omega$ such that $\Pi_\omega$ exists the
variance $\mathop{\rm Var}\nolimits^{\Phi}_{{\Lambda_\omega}}(f)$ is computed with boundary
conditions ($\tau$ outside $\Lambda$ and $\omega_{\Lambda\setminus \Lambda_\omega}$) which belong to ${\rm
  Max}_{\Lambda_\omega}$. Therefore we can bound it from above using the
Poincar\'e inequality by 
\begin{equation*}
\mathop{\rm Var}\nolimits^{\Phi}_{{\Lambda_\omega}}(f)\le   \gamma(\Lambda_\omega){\mathcal D}^{\Phi}_{\Lambda_\omega}(f) \le \gamma( B_1\cup I) {\mathcal D}^{\Phi}_{\Lambda_\omega}(f)
\end{equation*} 
where we used Lemma \ref{Mon} together with the observation that
$\Lambda_\omega\subset  B_1\cup I=\Lambda_1$. In conclusion 
\begin{gather}
\mu^{\Phi,\tau}_\Lambda\Bigl({1 \mskip -5mu {\rm I}}_{\{\exists\, \Pi_\omega \text{ in }I\}}
\mu^{\Phi}_{\Lambda_\omega}\bigl(\mathop{\rm Var}\nolimits^{\Phi}_{{ B_1}}(f)\bigr)\Bigr) \nonumber\cr
 \le \gamma(\Lambda_1) \mu^{\Phi,\tau}_\Lambda\Bigl({1 \mskip -5mu {\rm I}}_{\{\exists\, \Pi_\omega \text{
     in }I\}}{\mathcal D}^{\Phi}_{\Lambda_\omega}(f)\Bigr) \nonumber\cr
\le \gamma(\Lambda_1) \mu^{\Phi,\tau}_\Lambda\Bigl({1 \mskip -5mu {\rm I}}_{\{\exists\, \Pi_\omega \text{
     in }I\}}\sum_{x\in
  \Lambda_\omega}c_{x,\Lambda_\omega}\mathop{\rm Var}\nolimits_x^{\Phi}(f)\Bigr)\nonumber\cr
\le \gamma(\Lambda_1) \mu^{\Phi,\tau}_\Lambda\Bigl(\sum_{x\in  \Lambda_1}c_{x,\Lambda}\mathop{\rm Var}\nolimits_x^{\Phi}(f)\Bigr)
\label{B}\end{gather}
because, by construction, for every $\omega$ such that there exists $\Pi_\omega$
in $I$  
\begin{equation}
  \label{C}
c_{x,\Lambda_\omega}(\omega)=c_{x,\Lambda}(\omega)\quad \forall x\in \Lambda_\omega\,.
\end{equation} 
  
If we finally plug (\ref{B}) into the r.h.s. of (\ref{A})
and recall that $\Lambda_1\in {\mathcal F}_{k-1}$, 
we obtain
\begin{gather}
  \mu^{\Phi,\tau}_\Lambda\Bigl(c_1\mathop{\rm Var}\nolimits^{\Phi}_{{ B_1}}(f)\Bigr)
\le 
\gamma_{k-1}\,\mu^{\Phi,\tau}_\Lambda\bigl(\sum_{x\in
   \Lambda_1}c_{x,\Lambda}\mathop{\rm Var}\nolimits^{\Phi}_{x}(f)\bigr)
\label{D}
\end{gather}

In conclusion we have shown that
\begin{equation}
  \label{eq:2}
  \mathop{\rm Var}\nolimits^{\Phi,\tau}_\Lambda(f)\le
  \gamma^{(k)}_{\rm block}\gamma_{k-1}\Bigl({\mathcal D}^{\Phi,\tau}_\Lambda(f)+\sum_{x\in\D}\mu^{\Phi,\tau}_\Lambda\bigl(c_{x,\Lambda}\mathop{\rm Var}\nolimits_x(f)\bigr)\Bigr)
\end{equation}
Averaging over the $s_k= \lfloor l_k^{1/3} \rfloor$ possible choices of
the sets $\Lambda_1,\Lambda_2$ gives
\begin{equation}
  \label{eq:3}
  \mathop{\rm Var}\nolimits_\Lambda(f)\le \gamma^{(k)}_{\rm block}\gamma_{k-1}(1+\frac{1}{s_k}){\mathcal D}_\Lambda(f)
\end{equation}
which implies that
\begin{gather}
  \label{eq:4}
  \gamma_k\le (1+\frac{1}{s_k})\gamma^{(k)}_{\rm block}\gamma_{k-1}\hfill\break
\le \gamma_{k_0}\ \prod_{j=k_0}^k (1+\frac{1}{s_j})\gamma^{(j)}_{\rm block}
\end{gather}
where $k_0$ is the smallest integer such that $\gamma^{(k_0)}_{\rm
  block}<2$. 
If we now recall the expression \eqref{eq:inter5} for $\gamma^{(j)}_{\rm
  block}$ together with Lemma \ref{SM} and \ref{pippo}, we immediately conclude
that the product $\prod_{j=k_0}^\infty \gamma^{(j)}_{\rm
  block}(1+\frac{1}{s_j})$ is bounded.
\qed \end{proof}

\section{One spin facilitated model on a general graph}
\label{FA-1f general}
In this section we prove our second set of new results by examining the \emph{one spin facilitated} model
(FA-1f in short) on a general connected graph ${\mathcal G}=(V,E)$. Our
motivation comes from some  unpublished speculation by D. Aldous  \cite{Aldous2} that, in
this general setting, the FA-1f may
serve as an algorithm for information storage in dynamic graphs.   

We begin by discussing the finite setting.  Let $r$ be one of the
vertices and ${\mathcal T}$
be a rooted spanning tree of ${\mathcal G}$ with root $r$. On $\Omega=\{0,1\}^V$ consider the FA-1f constraints:
\begin{equation}
  \begin{cases}
  c_{x,{\mathcal G}}(\omega)=
    1 & \text{if $\omega_y=0$ for some neighbor $y$ of $x$} \label{FA1f-G}   \cr
 0 & \text{otherwise}
   \end{cases}
   \end{equation}
and let $\hat c_{x,{\mathcal G}}=c_{x,{\mathcal G}}$ if $x\neq r$ and $\hat
c_{r,{\mathcal G}}\equiv 1$. Let $\hat {\mathcal L}$ be the
corresponding  Markov generator and  notice that associated Markov
chain is ergodic since the vertex $r$ is unconstrained.  For shortness
we will refer in the sequel to $\hat {\mathcal L}$ as
the (${\mathcal G}$, $r$, FA-1f) model. Our first result
reads as follows.

\begin{Theorem}
\label{Fa-1f generale}
  \begin{equation*}
    \mathop{\rm gap}\nolimits({\mathcal G},r,\text{\rm FA-1f})\ge \mathop{\rm gap}\nolimits({\mathbb Z},\text{\rm East})
  \end{equation*}
\end{Theorem}
\begin{proof}
By monotonicity $\hat c_{x,{\mathcal G}}(\omega)\ge \hat c_{x,{\mathcal T}}(\omega)$ and therefore
$\mathop{\rm gap}\nolimits({\mathcal G},r,\text{FA-1f})\ge \mathop{\rm gap}\nolimits({\mathcal T}, r,\text{FA-1f})$.  We can push the monotonicity
argument a bit further and consider the following $\left({\mathcal T},r,\text{\rm East}\right)$ model:
\begin{equation}
  \tilde c_{x,{\mathcal T}}(\omega)=
  \begin{cases}
    1 & \text{if either $x=r$ or $\omega_y=0$, where $y$ is the ancestor (in
      ${\mathcal T}$) of
      $x$}\label{East-T}\cr
    0 & \text{otherwise}
  \end{cases}
\end{equation}
Clearly $\hat c_{x,{\mathcal T}}(\omega)\ge \tilde c_{x,{\mathcal T}}(\omega)$ and therefore
$\mathop{\rm gap}\nolimits({\mathcal G},r,\text{\rm FA-1f})\ge \mathop{\rm gap}\nolimits({\mathcal T},r,\text{East})$. We will now proceed
to show that
\begin{equation}
  \label{eq:2bis}
\mathop{\rm gap}\nolimits({\mathcal T},r,\text{East})\ge \mathop{\rm gap}\nolimits({\mathbb Z},\text{East})
\end{equation}
If all the vertices of ${\mathcal T}$ have degree $2$ with the exception of the
root and the leaves, {\it i.e. \ } if ${\mathcal T}\subset {\mathbb Z}$, then \eqref{eq:2bis} follows
from  \cite[Lemma 2.11]{noi}. Thus let us assume that there
exists $x\in{\mathcal T}$ with $\D_x\ge 3$ and let us order the vertices of
${\mathcal T}$ by first assigning some
arbitrary order to all vertices belonging to any given layer ($\equiv$
same distance from the root) and then declaring $x<y$ iff either
$d(x,r)<d(y,r)$ or $d(x,r)=d(y,r)$ and $x$ comes before $y$ in the order
assigned to their layer. Let $v$ be equal to the
root if $\D_r\ge 2$ or equal to the first descendant of $r$ with
degree $\D_v\ge 3$ otherwise and let $\Gamma_v=\{r,v_1,\dots,v_k,v\}$ be the path in ${\mathcal T}$ leading from $r$
to $v$. Let $a$ be a child of $v$ and let ${\mathcal T}_a=(V_a,E_a)$ be the subtree of ${\mathcal T}$
rooted in $a$. Finally we denote by $A$ and $B$ the two subgraphs of
${\mathcal T}$: $A:=\Gamma_v\cup {\mathcal T}_a$, $B:={\mathcal T}\setminus {\mathcal T}_a$.
(see Fig \ref{Fig1}).

\begin{figure}[ht]
\label{Fig1}
\centering
\psfrag{a}{$a$}
\psfrag{A}{$A$}
\psfrag{B}{$B$}
\psfrag{r}{$r$}
\psfrag{v}{$v$}
\psfrag{v1}{$v_1$}
\psfrag{gv}{$\Gamma_v$}
\includegraphics[width=.5\columnwidth]{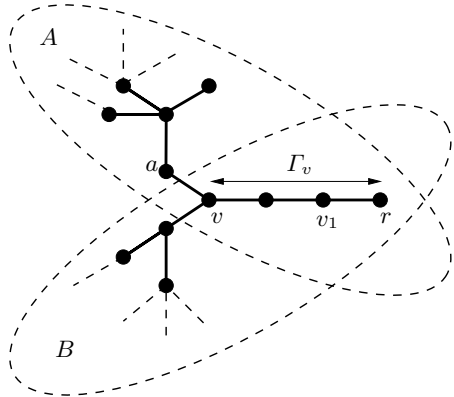}
\caption{The subtrees A and B}
\end{figure}

\begin{Lemma}
\label{lemma East}
\begin{equation}
  \label{eq:1}
\mathop{\rm gap}\nolimits({\mathcal T},r,\text{\rm East})\ge
\min\left(\mathop{\rm gap}\nolimits(A,r,\text{\rm East}),\mathop{\rm gap}\nolimits(B,r,\text{\rm East})\right)
\end{equation}
\end{Lemma}
By recursively applying  the above result to $A$ and $B$ separately, we immediately reduce
ourselves to the case of a tree ${\mathcal T}'\subset {\mathbb Z}$ and the proof of the
theorem is complete.

\qed 
\end{proof}

\begin{proof}[of Lemma \ref{lemma East}]
In $L^2(\Omega,\mu)$ consider the set ${\mathcal H}_B$ of functions $f$ that do not
depend on $\omega_x, \ x\in {\mathcal T}_a$.  Because of the choice of the
constraints $\tilde c_{x,{\mathcal T}}(\omega)$, ${\mathcal H}_B$ is an invariant subspace for the
generator of the $\left({\mathcal T},r,\text{\rm East}\right)$ model and 
\begin{equation}
  \label{eq:4bis}
\inf_{\substack {f\in {\mathcal H}_B \hfill\break \mu(f)=0}}\frac{\tilde {\mathcal D}(f)}{\mathop{\rm Var}\nolimits(f)}= \mathop{\rm gap}\nolimits(B,r,\text{East})
\end{equation}
Let us now consider the orthogonal subspace ${\mathcal H}^{\perp}_B$. Any zero
mean element
$f\in {\mathcal H}^{\perp}_B$ satisfies $\mu_{{\mathcal T}_a}(f)=0$ and therefore we can write
\begin{gather}
  \mathop{\rm Var}\nolimits(f)=\mu\left(\mathop{\rm
      Var}\nolimits_{{\mathcal T}_a}(f)\right) \le\mu\left(\mathop{\rm
      Var}\nolimits_{A}(f)\right) \nonumber\cr \le
  \mathop{\rm gap}\nolimits(A,r,\text{East})^{-1}\sum_{x\in A}\mu\left(\tilde c_{x,A}\mathop{\rm Var}\nolimits_x(f)\right)
  \nonumber \cr
\le   \mathop{\rm gap}\nolimits(A,r,\text{East})^{-1}\tilde {\mathcal D}(f)
  \label{eq:3bis}
\end{gather}
where the first inequality follows from convexity of the variance and
the second one is nothing but the Poincar\'e inequality for the East
model in $A$. The proof of the Lemma follows at once from \eqref{eq:4bis}, \eqref{eq:3bis}.
\qed 
\end{proof}

Theorem \ref{Fa-1f generale} has two consequences that will be the
content of the following Theorems. The first one deals with the case
of an infinite graph. The second one deals with the FA-1f model on
general graph ${\mathcal G}$ without the special unblocked vertex $r$
but with the Markov chain restricted to a suitable ergodic
component. 

\begin{Theorem}
Let ${\mathcal G}_\infty$ be an infinite connected graph of bounded degree and
let ${\mathcal L}$ be the generator of the FA-1f model on ${\mathcal G}_\infty$ with
constraints $\{c_{x,{\mathcal G}_\infty},\ x\in V_\infty\}$, {\it i.e. \ } no apriori unblocked vertex. Then 
  \begin{equation*}
    \mathop{\rm gap}\nolimits({\mathcal G}_\infty,\text{\rm FA-1f})\ge \mathop{\rm gap}\nolimits({\mathbb Z},\text{\rm East})
  \end{equation*}
\end{Theorem}
\begin{proof}
The proof combines Theorem \ref{Fa-1f generale} together with the finite subgraph
approximation described in section \ref{main questions}.  
\qed \end{proof}

\begin{Theorem}
Let ${\mathcal G}$ be as in Theorem  \ref{Fa-1f generale} and let ${\mathcal L}^+$ be
the FA-1f generator with constraints $\{c_{x,{\mathcal G}}\}_{x\in V}$ on the restricted configuration space $\Omega^+:=\{\eta\in
\Omega:\ \sum_{x\in V}(1-\eta_x)\ge 1\}$ equipped with the reversible measure
$\mu^+:=\mu(\cdot\thinspace | \thinspace\Omega^+)$. Then
  \begin{equation*}
\mathop{\rm gap}\nolimits({\mathcal L}^+)  \ge   \frac 12 \mathop{\rm gap}\nolimits({\mathbb Z},\text{\rm East})\mu(\Omega^+)
  \end{equation*}
\end{Theorem}
\begin{proof}
  As in the proof of Theorem \ref{Fa-1f generale} we can safely assume
  that ${\mathcal G}$ is a tree ${\mathcal T}$ with root $r\in V$. We extend any $f:\Omega^+\mapsto
  {\mathbb R}$ to a function $\tilde f$ on $\Omega$ by setting $\tilde f(\eta_y=1\ \forall y)\equiv
  f(\eta_y=1\ \forall y\neq r,\ \eta_r=0)$. Using Theorem
  \ref{Fa-1f generale}, we then write
\begin{gather}
\mathop{\rm Var}\nolimits^+(f)=\mathop{\rm Var}\nolimits^+(\tilde f) \le \left(\mu(\Omega^+)\right)^{-2}\mathop{\rm Var}\nolimits(\tilde f)
\nonumber
\cr
\le \left(\mu(\Omega^+)\right)^{-2} \mathop{\rm gap}\nolimits({\mathcal T},r, \text{East})^{-1}\sum_{x}\mu\left(\hat
c_{x,{\mathcal T}} \mathop{\rm Var}\nolimits_x(\tilde f)\right) 
\label{C.1}
\end{gather}
where the constraints $\{\hat c_{x,{\mathcal T}}\}_{x\in {\mathcal T}}$ have been defined right after \eqref{FA1f-G}. 

Let us examine a generic term $\mu\left(\hat c_{x,{\mathcal T}} \mathop{\rm Var}\nolimits_x(\tilde f)\right)$ with
$x\neq r$. Remember that $\hat c_{x,{\mathcal T}}=c_{x,{\mathcal T}}$ and moreover $c_{x,{\mathcal T}}(\eta)=0$
if $\eta_y=1$ for all $y\neq x$. Furthermore, for any $\eta$ such that there
exists $y\neq x$ with $\eta_y=0$, $ \mu^+(\eta_x=1\thinspace | \thinspace\{\eta_y\}_{y\neq x})=p
$. In conclusion we have shown that
\begin{equation}
  \label{C.2}
 \mu\left(\hat c_{x,{\mathcal T}} \mathop{\rm Var}\nolimits_x(\tilde f)\right)=\mu(\Omega^+)\mu^+\left(c_{x,{\mathcal T}}
 \mathop{\rm Var}\nolimits^+_x(f)\right)\qquad \forall x\neq r
\end{equation}
We now examine the dangerous term $\mu\left(\hat c_{r,{\mathcal T}} \mathop{\rm Var}\nolimits_r(\tilde
f)\right)=\mu\left(\mathop{\rm Var}\nolimits_r(\tilde f)\right)$. Because of the definition of $\tilde f$
we can safely rewrite it as 
\begin{gather*}
\mu\left(\mathop{\rm Var}\nolimits_r(\tilde f)\right)=\mu\left(\chi_{\{\exists\, y\neq r:\,
  \eta_y=0\}}\mathop{\rm Var}\nolimits_r(f)\right)
\end{gather*}
Let us order the vertices of the tree ${\mathcal T}$ starting from the \emph{furthermost
ones} by first assigning some arbitrary order to all vertices belonging
to any given layer ($\equiv$ same distance from the root) and then
declaring $x<y$ iff either $d(x,r)>d(y,r)$ or $d(x,r)=d(y,r)$ and $x$
comes before $y$ in the order assigned to their layer. Next, for any
$\eta$ such that $\eta_y=0$ for some $y\neq r$, define $\xi=\min\{y:\
\eta_y=0\}$ and let ${\mathcal T}_\xi:=\{z\in {\mathcal T}:\ z>\xi\}$ (see Fig \ref{Tcsi}). 

\begin{figure}[ht]
\psfrag{1}{\tiny $1$}
\psfrag{2}{\tiny$2$}
\psfrag{3}{\tiny$3$}
\psfrag{4}{\tiny$4$}
\psfrag{5}{\tiny$5$}
\psfrag{6}{\tiny$6...$}
\psfrag{16}{\tiny$16$}
\psfrag{17}{\tiny$17$}
\psfrag{18}{\tiny$18$}
\psfrag{22}{\tiny$22$}
\psfrag{24}{\tiny$24$}
\psfrag{25}{\tiny$25$}
\psfrag{26}{\tiny$26$}
\psfrag{27}{\tiny$27$}
\psfrag{28}{\tiny$28$}
\psfrag{29}{\tiny$29$}
\psfrag{30}{\tiny$30$}
\psfrag{31}{\tiny$31$}
\psfrag{r}{$r$}
\psfrag{v}{$v$}
\psfrag{xi}{$\xi$}
\includegraphics[width=.99\columnwidth]{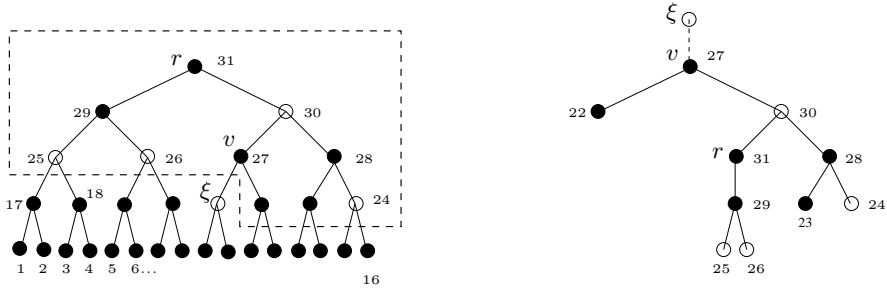}
\caption{An example of a tree $\mathcal{T}$ on the left, with a choice
of an ordering.
The dotted line delimitates the subtree $\mathcal{T}_\xi$ that we
reproduce on the right, with
root $v$.}
\label{Tcsi}
\end{figure}

Notice that the
subgraph ${\mathcal T}_\xi$ is again a tree and we define its root to be the
ancestor $v$ of $\xi$ in ${\mathcal T}$. Then, using convexity of the variance,
we can write
\begin{gather*}
  \mu\left(\chi_{\{\exists \, y\neq
    r:\,\eta_y=0\}}\mathop{\rm
    Var}\nolimits_r(f)\right)=\mu\left(\chi_{\xi\neq
    r}\mu\left(\mathop{\rm Var}\nolimits_r(f)\thinspace | \thinspace
  \xi\right)\right)\hfill\break
\le \mu\left(\chi_{\xi\neq r}\mathop{\rm Var}\nolimits_{{\mathcal T}_\xi}(f)\right)
\end{gather*}
In order to bound from above $\mathop{\rm Var}\nolimits_{{\mathcal T}_\xi}(f)$ we apply the Poincar\'e
inequality in ${\mathcal T}_\xi$ with constraints $\{\hat c_{z,{\mathcal T}_\xi}\}$ and root
$v$ together with
Theorem \ref{Fa-1f generale}:
$$
\mathop{\rm Var}\nolimits_{{\mathcal T}_\xi}(f)\le \mathop{\rm gap}\nolimits({\mathbb Z},\text{East})^{-1}
\sum_{z\in {\mathcal T}_\xi}\mu_{{\mathcal T}_\xi}\left( \hat c_{z,{\mathcal T}_\xi}\mathop{\rm Var}\nolimits_z(f)\right)
$$
Notice that, by construction, $\hat c_{z,{\mathcal T}_\xi}(\eta)=c_{z,{\mathcal T}}(\eta)$ for any
  $z\in {\mathcal T}_\xi$, including the root $v$ of ${\mathcal T}_\xi$ where
  $\hat c_{v,{\mathcal T}_\xi}(\eta)=1$ by definition and
    $c_{v,{\mathcal T}}(\eta)=1$ because $\eta_\xi=0$. 
Putting all together we conclude that
\begin{gather}
   \mu\left(\chi_{\{\exists \, y\neq
    r:\,\eta_y=0\}}\mathop{\rm Var}\nolimits_r(f)\right) \le \mathop{\rm gap}\nolimits({\mathbb Z},\text{East})^{-1}\sum_{x\in {\mathcal T}}\mu\left(\chi_{\{\exists\, y\neq r:\,
  \eta_y=0\}}c_{x,{\mathcal T}}\mathop{\rm Var}\nolimits_x(f) \right)
\nonumber \cr
 \le 
\mathop{\rm gap}\nolimits({\mathbb Z},\text{East})^{-1} \mu(\Omega^+)\sum_{x\in {\mathcal T}}\mu^+\left(c_{x,{\mathcal T}}\mathop{\rm Var}\nolimits^+_x(f) \right) 
\label{C.3}
\end{gather}
where we have used once more the observation before \eqref{C.2} to write 
$$
c_{x,{\mathcal T}}\mathop{\rm Var}\nolimits_x(f)=c_{x,{\mathcal T}}\mathop{\rm Var}\nolimits^+_x(f).
$$
If we now combine \eqref{C.1}, \eqref{C.2} and \eqref{C.3} together we
get
\begin{equation*}
\mathop{\rm Var}\nolimits^+(f) \le 2\left(\mathop{\rm gap}\nolimits({\mathbb Z},\text{East})\mu(\Omega^+)\right)^{-1}  \sum_{x\in {\mathcal T}}\mu^+\left(c_{x,{\mathcal T}}\mathop{\rm Var}\nolimits^+_x(f) \right)
\end{equation*}
and the proof is complete.
\qed \end{proof}
\begin{acknowledgement}
F. Martinelli would like to warmly thank Roman Koteck\'y for the very nice
invitation to lecture at the Prague Summer School on \emph{Mathematical
  Statistical Mechanics}.  We also acknowledge P.Sollich for useful discussions
and comments on the topics of this paper. 
\end{acknowledgement}

\bibliographystyle{amsplain}
\bibliography{ref}


\printindex
\end{document}